\documentclass[12pt,leqno]{article}
\usepackage{geometry}
\usepackage{bm}
\usepackage{xcolor}
\usepackage{mathtools}

 \usepackage[hidelinks]{hyperref}

\usepackage{amsmath,amsthm,amsfonts,amssymb,latexsym,amscd,enumerate}
\usepackage{slashed}

\usepackage{palatino}
\usepackage[mathcal]{euler}

\usepackage{comment}
\usepackage{tikz-cd}
\usetikzlibrary{decorations.markings}

\usepackage{xy}
\xyoption{all}

\numberwithin{equation}{subsection}
\swapnumbers

\newtheorem{theorem}[equation]{Theorem}
\newtheorem*{theorem*}{Theorem}

\newtheorem{corollary}[equation]{Corollary}
\newtheorem{proposition}[equation]{Proposition}
\newtheorem*{proposition*}{Proposition}
\newtheorem{lemma}[equation]{Lemma}
\newtheorem*{conjecture*}{Conjecture}

\theoremstyle{definition}
\newtheorem{definition}[equation]{Definition}

\newtheorem{example}[equation]{Example}

\newtheorem{remark}[equation]{Remark}

\newcommand{\R}{\mathbb{R}}
\newcommand{\C}{\mathbb{C}} 

\newcommand{\Z}{\mathbb{Z}}

\renewcommand{\Re}{\mathrm{Re}}
\renewcommand{\Im}{\mathrm{Im}}

\DeclareMathOperator{\mult}{mult}
 
\DeclareMathOperator{\Ad}{Ad} 

\DeclareMathOperator{\Hom}{Hom}

\DeclareMathOperator{\Ind}{Ind}
\DeclareMathOperator{\InfCh}{Inf\,Ch}
\DeclareMathOperator{\ImInfCh}{Im\,Inf\,Ch} 
\DeclareMathOperator{\ReInfCh}{Re\,Inf\,Ch}

\newcommand{\lie}{\mathfrak}

\DeclareMathOperator{\Sym}{Sym}

\newcommand{\GG}{\pmb{G}}

\newcommand{\CC}{\mathbb{C}}

\newcommand{\temp}{\mathrm{tempered}}
\newcommand{\dom}{\mathrm{dom}}

\usepackage{cancel} 

\title{Operator K-Theory and Tempiric Representations}

\author{Jacob Bradd,  Nigel Higson and Robert Yuncken%
\thanks
{
This research was undertaken during a visit to the Sydney Mathematical Research Institute (SMRI); 
N.~Higson and R.~Yuncken are grateful for the support and hospitality of the SMRI, the University of New South Wales, and the University of Wollongong.  
J.~Bradd and R.~Yuncken were supported by project OpART of the Agence Nationale de la Recherche (ANR-23-CE40-0016).  This article is based upon work from COST Action CaLISTA CA21109 supported by COST (European Cooperation in Science and Technology): \url{www.cost.eu}.  N.~Higson was supported  by the NSF grant DMS-1952669.
}
}
\date{}

\begin{document}

\maketitle

  \abstract{David Vogan proved that if $G$ is a real reductive group, and if $K$ is a maximal compact subgroup of $G$, then  every irreducible representation  of $K$  is included as a minimal $K$-type in  precisely one tempered, irreducible unitary  representation   of $G$ with real infinitesimal character, and that moreover  it is included there with multiplicity one and is the unique minimal $K$-type in that representation.   We shall prove that  the Connes-Kasparov isomorphism in operator $K$-theory is equivalent to a $K$-theoretic version of Vogan's result.}

  
 \section{Introduction} 
  
Alexandre Afgoustidis has suggested   the name \emph{tempiric representation} for any irreducible, tempered, unitary, real infinitesimal character representation of a real reductive group.   The goal of this paper is to use the Connes-Kasparov isomorphism in operator $K$-theory to prove a $K$-theoretic version of the following theorem of David Vogan: if $G$ is a real reductive group, and if $K$ is a maximal compact subgroup of $G$, then  every irreducible representation  of $K$ is included as a minimal $K$-type in    precisely one tempiric representation of $G$, and  moreover it is included there  with multiplicity one and is the unique minimal $K$-type in that representation.  See \cite{Voganbook} for the representation-theoretic ideas involved, and \cite[Thm.~1.2]{Vogan07} for the precise statement of the theorem.  See   the introduction to \cite{ClareHigsonSongTang24} and the paper \cite{LafforgueICM} for  background information about the Connes-Kasparov isomorphism from the points of view of representation theory and operator $K$-theory, respectively. 

In fact we shall also show that, in the reverse direction, our $K$-theoretic version of Vogan's theorem implies the Connes-Kasparov isomorphism.  The main result of this paper therefore has two parts.  It is stated precisely in Section~\ref{sec-vogans-thm} below, but informally, it is as follows:

\begin{theorem*}
The following two assertions are equivalent: 
\begin{enumerate}[\rm (i)]

\item For every real reductive group $G$, the Connes-Kasparov morphism in operator $K$-theory for the group $G$ is an isomorphism of abelian groups.

\item  For every real reductive group $G$ and maximal compact subgroup $K$, there is an isomorphism of abelian groups  from the representation ring $R(K)$ to the free abelian group on the tempiric representations of $G$ that maps an irreducible representation $\tau \in \widehat K$ to the combination $\sum_\pi \mathrm{mult} ( \tau, \pi)\cdot \pi $ whose coefficients are the multiplicities with which $\tau$ occurs in the tempiric representations of $\pi$.

\end{enumerate}

\end{theorem*}

The proof of our theorem  is based on the notion of the imaginary part of the infinitesimal character of an irreducible unitary representation, reviewed in Section~\ref{sec-imaginary-part-of-inf-ch}.  The importance of this concept has been emphasized by David Vogan \cite{Vogan00}. We shall use it in Section~\ref{sec-reduced-c-star-algebra-and-im-inf-ch} to filter the reduced $C^*$-algebra of a real reductive group by ideals whose associated subquotients have an especially simple structure.  

We shall do something similar in Section~\ref{sec-cartan-motion-group}  for the Cartan motion group associated to a real reductive group.  This is much easier.  Then in Section~\ref{sec-deformation-family} we shall repeat the construction for all the groups in a smooth family of groups that interpolates between a reductive group and its Cartan motion group (the deformation to the normal cone construction).  By good fortune, perhaps, the $K$-theory groups here can be computed in a very simple and direct manner from the morphisms defined in part (ii) of the theorem above. Our main results follow from this and Alain Connes' reformulation of the Connes-Kasparov isomorphism as an assertion about the operator $K$-theory of the smooth family \cite[Prop.~9,~p.141]{ConnesNCG}.

Our results have much in common with the so-called Mackey bijection between irreducible tempered unitary representations of a real reductive group and the irreducible unitary representations of its Cartan motion group. In particular, the work of Afgoustidis on the Mackey bijection \cite{AfgoustidisConnesKasparov,AfgoustidisMackeyBijection} is extremely closely related (all of the necessary ideas already appear in these two papers of Afgoustidis; we have simply arranged them in a different way). Other very relevant works are \cite{DeBelloHigson24}, \cite{HigsonRoman20} and \cite{ClareHigsonRoman24}.


\section{Imaginary part of the infinitesimal character}
\label{sec-imaginary-part-of-inf-ch}
In this section we shall gather most of the information from Lie theory and representation theory that we shall need in the paper. 

\subsection{Real reductive groups}
Throughout the paper, we shall use the term \emph{real reductive group} in  the sense that is used in  Vogan's monograph \cite[Ch.\,0,\,Sec.\,1]{Voganbook}. This means among other things that a real reductive group $G$ is required to be linear, but it need not be connected, although there are restrictions in that   regard.  Examples include the group of real points of any connected reductive algebraic group defined over $\R$, including for instance $SL(n,\R)$, $Sp(2n,\R)$, etc.

We shall  use the following structure, as discussed in \cite[Ch.\,0,\,Sec.\,1]{Voganbook}. Every real reductive  group  $G$ is equipped with a \emph{Cartan involution}
\[
\theta \colon G \longrightarrow G
\]
whose fixed subgroup is precisely $K$. We shall write the \emph{Cartan decomposition} of the Lie algebra $\mathfrak{g}$ into $\pm 1$ eigenspaces of $\theta$ as $\mathfrak{g} = \mathfrak{k} \oplus \mathfrak{s}$
(we shall not follow the custom of indicating real Lie algebras by a subscript $0$; instead we shall indicate complexifications by a subscript $\C$). We can and will equip $\mathfrak{g}$ with an invariant nondegenerate symmetric bilinear form 
\begin{equation}
    \label{eq-def-of-b}
B\colon \mathfrak{g}\times \mathfrak{g} \longrightarrow \R
\end{equation}
that is positive-definite on $\mathfrak{s}$ and negative-definite on $\mathfrak{k}$, and that is invariant under $\theta$.   When we speak of inner products on subspaces of $\mathfrak{k}$ or of $\mathfrak{s}$, we shall always mean the positive-definite inner products on these subspaces obtained from $B$.

We shall fix throughout a maximal abelian subspace $\mathfrak{a}_{\min}$ of $\mathfrak{s}$.  We shall fix a system of positive restricted roots  
\begin{equation}
    \label{eq-fixed-system-of-positive-roots}
\Delta^+(\mathfrak{g},\mathfrak{a}_{\min}) \subseteq \Delta(\mathfrak{g},\mathfrak{a}_{\min}) ,
\end{equation}
along with the associated \emph{Iwasawa decomposition}
$G = K A_{\min}N_{\min}$ and \emph{minimal parabolic subgroup} 
\begin{equation}
    \label{eq-fixed-minimal-parabolic-subgroup}
P_{\min} = M_{\min}A_{\min}N_{\min}
\end{equation}
(see for instance \cite[Ch.\,VI]{KnappBeyond}). Here  $M_{\min}$ is  the centralizer of  $\mathfrak{a}_{\min}$ in $K$. 

We shall always denote by $S\subseteq \Delta^{+}(\mathfrak{g},\mathfrak{a}_{\min})$ the  set   of simple restricted roots. The \emph{standard parabolic subgroups} of $G$ are the parabolic subgroups that include the chosen minimal parabolic subgroup $P_{\min}$, and they are in bijection with  subsets  $I\subseteq S$, as follows:
\begin{equation}
    \label{eq-standard-parabolic-subgroups}
P_I=M_IA_IN_I,  \quad \mathfrak{a}_I = \bigcap \{ \, \operatorname{ker}(\alpha) : \alpha \in I \,\}.
\end{equation}
Here, the Levi subgroup $M_IA_I$ is the centralizer of $\mathfrak{a}_I$ in $G$, and $N_I$  is determined by the requirement that $P_I$ be standard.  See for instance \cite[Sec.\,VII.7]{KnappBeyond}.

Finally, if $P=MAN$ is a standard parabolic subgroup,   then we shall always regard the dual vector space $\mathfrak{a}^*$ as a vector subspace of  $\mathfrak{a}_{\min}^*$ by means of the map
\begin{equation}
    \label{eq-inclusion-of-a-spaces}
\mathfrak{a}^* \longrightarrow \mathfrak{a}_{\min}^*
\end{equation}
that maps a linear functional on $\mathfrak{a}$ to its unique extension to a linear functional on $\mathfrak{a}_{\min}$ that vanishes on the orthogonal complement of $\mathfrak{a}$ in $\mathfrak{a}_{\min}$. The orthogonal complement may also be described as the intersection of $\mathfrak{a}_{\min}$ with $\mathfrak{m}_I$.  When $\mathfrak{a}=\mathfrak{a}_I$, the  image of \eqref{eq-inclusion-of-a-spaces}
 is the space 
\begin{equation}
    \label{eq-characterization-of-a-i-star}
 \mathfrak{a}_I^* \cong \{ \, \gamma \in \mathfrak{a}_{\min}^* : \langle \alpha, \gamma\rangle = 0\,\, \forall \alpha\in I \,\} .
\end{equation}

\subsection{Imaginary part of the infinitesimal character}
 
\begin{definition} 
Denote by $\mathcal{Z}(\mathfrak{g_\C})$ the center of the enveloping algebra of $\mathfrak{g_\C}$, the complexification of the real Lie algebra $\mathfrak{g}$. If $\mathfrak{h_\C}$ is any Cartan subalgebra of $\mathfrak{g_\C}$, then we shall denote by  
\begin{equation}
    \label{eq:Harish-Chandra}
\xi_{\mathfrak{h}}\colon \mathcal{Z}(\mathfrak{g_\C}) \stackrel\cong \longrightarrow \operatorname{Sym}(\mathfrak{h_\C})^{W(\mathfrak{g_\C},\mathfrak{h_\C})} 
\end{equation}
the Harish-Chandra isomorphism (see e.g.~\cite[Sec.~V.5]
{KnappBeyond}).
\end{definition} 

On the right-hand side of the Harish-Chandra isomorphism is the Weyl group-invariant part of the symmetric algebra of $\mathfrak{h_\C}$, or equivalently the algebra of  Weyl group-invariant polynomial functions on $\mathfrak{h_\C}^*$.   
Every character of $\operatorname{Sym}(\mathfrak{h_\C})^{W(\mathfrak{g_\C},\mathfrak{h_\C})}$, or in other words every algebra homomorphism 
\[
\operatorname{Sym}(\mathfrak{h_\C})^{W(\mathfrak{g}_\C,\mathfrak{h}_\C)} \longrightarrow \C ,
\]
is given by evaluation   at some weight $\lambda \in \mathfrak{h}_\C^*$. Two weights   correspond to same character if and only if they lie in the same $W(\mathfrak{g}_\C,\mathfrak{h}_\C)$-orbit.  

It follows that every character   
\[
\chi \colon \mathcal{Z}(\mathfrak{g}_\C)\longrightarrow \C
\]
may be written as a composition
\begin{equation}
    \label{eq-form-of-a-character}
 \mathcal{Z}(\mathfrak{g}_\C)   \stackrel{\xi_{\mathfrak{h}}}\longrightarrow \operatorname{Sym}(\mathfrak{h})^{W(\mathfrak{g}_\C,\mathfrak{h}_\C)} \stackrel{\mathrm{eval}_\lambda}\longrightarrow \C
\end{equation}
of the Harish-Chandra isomorphism $\xi_{\mathfrak{h}}$ and   evaluation at some $\lambda\in \mathfrak{h}_\C^*$,  which  unique up to the action of $W(\mathfrak{g}_\C,\mathfrak{h}_\C)$.  We shall write such a composition as 
\[
\chi = \xi_{\mathfrak{h},\lambda} \colon \mathcal{Z}(\mathfrak{g}_\C)\longrightarrow \C .
\]

Following Vogan \cite[Sec.\,2]{Vogan00}, we wish to identify some of these characters as \emph{real} and some as \emph{imaginary}.

\begin{definition} 
\label{def-real-and-imaginary-weights}
Let  $\mathfrak{h} = \mathfrak{t} \oplus \mathfrak{a}$   be a $\theta$-stable Cartan subalgebra of $\mathfrak{g}$, decomposed into its compact and noncompact parts.  The \emph{real part} of its complexification, $\mathfrak{h}_\C$,  is the real vector subspace  
\[
\Re(\mathfrak{h}_\C) =   i\mathfrak{t}\oplus \mathfrak{a}  \subseteq  \mathfrak{h}_\C .
\]
If the group $G$ is realized as a group of real matrices in such a way that the Cartan involution is $\theta(X) = - X^{\mathrm{transp.}}$, then these are the matrices in $\mathfrak{h}_\C$ with real eigenvalues. A weight $\lambda \in \mathfrak{h}_\C^*$ is \emph{real} if it is real-valued on $\Re(\mathfrak{h}_\C)$. Similarly it is  \emph{imaginary} if it is imaginary-valued on $\Re(\mathfrak{h}_\C)$.  
\end{definition}

\begin{lemma}[{\cite[Lemma\,2.6]{Vogan00}}]
\label{lem-canonical-real-part-of-h-mod-w}
    The action of the Weyl group $W(\mathfrak{g}_\C,\mathfrak{h}_\C)$  on $\mathfrak{h}_\C$ preserves the real subspace $\Re(\mathfrak{h}_\C)$. More generally, if $\Ad_g$ is any inner automorphism of the complex Lie algebra $\mathfrak{g}_\C$ that conjugates a Cartan subalgebra $\mathfrak{h}_{1,\C}$ to a second, $\mathfrak{h}_{2,\C}$, then $Ad_g$ conjugates $\Re(\mathfrak{h}_{1,\C})$ and $\Im(\mathfrak{h}_{1,\C})$ to $\Re(\mathfrak{h}_{2,\C})$ and $\Im(\mathfrak{h}_{2,\C})$, respectively.
\end{lemma}

\begin{definition}
If $\lambda \in \mathfrak{h}_\C^*$, then   we shall denote by 
\[
\lambda = \Re(\lambda) + \Im(\lambda)
\]
the (unique) decomposition of a weight $\lambda$ as a sum of real and imaginary weights.
\end{definition} 

It follows from the Lemma~\ref{lem-canonical-real-part-of-h-mod-w} that: 

\begin{lemma}[{\cite[Cor.~2.7]{Vogan00}}]
\label{lem-vogans-canonical-real-and-imaginary-parts}
Let $\mathfrak{h}$ be a $\theta$-stable Cartan subalgebra of $\mathfrak{g}$.
Let $\chi\colon \mathcal{Z}(\mathfrak{g}_\C)   \to \C$ be any character.  If $\chi$ is written in the form
\[
\chi = \xi_{\mathfrak{h},\lambda}\colon \mathcal{Z}(\mathfrak{g}_\C)   \longrightarrow \C
\]
for some  $\lambda \in \mathfrak{h}_\C^*$, then the characters 
\[
\xi_{\mathfrak{h},\Re(\lambda)}\colon \mathcal{Z}(\mathfrak{g}_\C)   \to \C
\quad \text{and} \quad \xi_{\mathfrak{h},\Im(\lambda)}\colon \mathcal{Z}(\mathfrak{g}_\C)   \to \C
\]
depend only on $\chi$, and not on the choices of $\mathfrak{h}$ and $\lambda$.
\end{lemma}

\begin{definition}
Given a character $\chi\colon \mathcal{Z}(\mathfrak{g}_\C)   \to \C$ as in Lemma~\ref{lem-vogans-canonical-real-and-imaginary-parts}, we shall write 
\[
\Re(\chi) =\xi_{\mathfrak{h},\Re(\lambda)}\colon \mathcal{Z}(\mathfrak{g}_\C)   \to \C
\quad \text{and} \quad \Im(\chi) =\xi_{\mathfrak{h},\Im(\lambda)}\colon \mathcal{Z}(\mathfrak{g}_\C)   \to \C .
\]
These are the \emph{real and imaginary parts} of  $\chi$.
\end{definition}

Now let $\pi$ be an irreducible unitary representation of $G$. 
Recall that, by Schur's Lemma, $\pi$ has an associated infinitesimal character
\[
\InfCh (\pi) : \mathcal{Z}(\lie{g}_\C) \to \C.
\]
From the above, we can form the real and imaginary parts of the infinitesimal character of $\pi$: 
\[
\operatorname{Re\,Inf\,Ch} (\pi)\colon \mathcal{Z}(\mathfrak{g}_\C)   \to \C \quad \text{and} \quad 
\ImInfCh (\pi) \colon \mathcal{Z}(\mathfrak{g}_\C)   \to \C.
\]

We shall need some information about the real and imaginary parts of the infinitesimal character of an irreducible, tempered unitary representation, starting with the following result:

\begin{theorem}
\label{thm-discrete-series-have-real-inf-ch}
    If $\pi$ is a discrete series representation of a real reductive group, then $\ImInfCh(\pi)=0$.
\end{theorem}

This follows from Harish-Chandra's classification of the discrete series via Harish-Chandra parameters (which determine the infinitesimal characters of the discrete series). See Remark~\ref{remark-appendix-harish-chandra-parameters} in the appendix.  

We shall use Theorem~\ref{thm-discrete-series-have-real-inf-ch} and the following well-known principle of Harish-Chandra and Langlands \cite{Langlands89} to say something about the infinitesimal character of a general irreducible tempered unitary representation.

\begin{theorem}[See for instance {\cite[Thm.~8.5.3]{KnappRepTheorySemisimpleGroups}}]
\label{thm-harish-chandra-principle-on-cuspidal-principal-series}
    Every irreducible tempered unitary representation of $G$ embeds into  some representation $\Ind_{P}^G \sigma \otimes \exp(i \nu)$, where $P=MAN$ is a parabolic subgroup of $G$, $\sigma $  is a discrete series representation of $M$, and $\nu\in \mathfrak{a}^*$.
\end{theorem}

It follows from Theorem~\ref{thm-harish-chandra-principle-on-cuspidal-principal-series} that every irreducible tempered unitary representation has the same infinitesimal character as some cuspidal principal series representatation $\Ind_{P}^G \sigma \otimes \exp(i \nu)$.  But the infinitesimal character of the latter may be readily computed as follows:
\begin{equation}
    \label{eq-inf-ch-of-principal-series}
\left \{ 
\begin{aligned} 
\ImInfCh \bigl ( \Ind_P^G \sigma\otimes \exp(i \nu) \bigr ) & = i \nu \in   i \mathfrak{a}^*  ,
\\
\ReInfCh \bigl ( \Ind_P^G \sigma\otimes \exp(i \nu) \bigr ) & = \InfCh(\sigma)  \in i\mathfrak{t}^* .
\end{aligned}
\right .
\end{equation}
See for instance Knapp \cite[Prop.\,8.22]{KnappRepTheorySemisimpleGroups}. The infinitesimal character in \eqref{eq-inf-ch-of-principal-series} is computed  using a $\theta$-stable Cartan subalgebra of the form $\mathfrak{h} = \mathfrak{t}\oplus \mathfrak{a}$ that is constructed from the parabolic subgroup $P=MAN$ in Theorem~\ref{thm-harish-chandra-principle-on-cuspidal-principal-series} by defining $\mathfrak{t}$ to be  the Lie algebra   of a maximal torus in $K \cap M$ (which by another well-known result of Harish-Chandra is a Cartan subalgebra of $\mathfrak{m}$) and of course defining $\mathfrak{a}$ to be the Lie algebra of $A$.  The infinitesimal character of the discrete series representation $\sigma$ lies in $i\mathfrak{t}^*$  thanks to Theorem~\ref{thm-discrete-series-have-real-inf-ch}.

We want to use a single $\theta$-stable Cartan subalgebra of $\mathfrak{g}$ to describe the infinitesimal characters of \emph{all} the irreducible tempered unitary representations of $G$.  To this end, we shall use the minimal parabolic subgroup that we fixed in \eqref{eq-fixed-minimal-parabolic-subgroup}.
Form the ``minimally compact'' Cartan subalgebra
\begin{equation}
    \label{eq-minimally-compact-cartan}
\mathfrak{h}_{\min} = \mathfrak{t}_{\min} \oplus \mathfrak{a}_{\min} ,
\end{equation}
where $\mathfrak{t}_{\min}$ is the Lie algebra of a maximal torus in $M_{\min}$.
It follows from  computations with  Cayley transforms (see for instance \cite[Sec.\,VI.7]{KnappBeyond}) that there is an inner automorphism of the complex Lie algebra $\mathfrak{g}_\C$ that conjugates the Cartan subalgebra $\mathfrak{h}_\C$ in \eqref{eq-inf-ch-of-principal-series} to $\mathfrak{h}_{\min,\C}$, and conjugates $\mathfrak{a}$ into $\mathfrak{a}_{\min}$.  This inner automorphism is not unique, of course, but the possible images of $\nu$ lie within a single orbit of the restricted Weyl group:

\begin{lemma}[{See \cite[Lemma~3.4]{Vogan00} and \cite[Cor.\,VII.8.9]{HelgasonDS}}]
\label{lem-helgason}
Two elements of $\mathfrak{a}_{\min}$ are conjugate in $\mathfrak{g}_{\C}$ by an element of $\mathrm{Inn}(\mathfrak{g}_\C)$ if and only if they are conjugate by an element of $W(\mathfrak{g},\mathfrak{a}_{\min})$.
\end{lemma}

We arrive at the following result:

\begin{theorem}[{\cite[Cor.~3.5(1)]{Vogan00}}]
\label{thm-vogan-and-maybe-knapp}
The imaginary part of the infinitesimal character of any irreducible tempered unitary representation $\pi$ may be written in the following form, using the  $\theta$-stable Cartan subalgebra $\mathfrak{h}_{\min}{=}\mathfrak{t}_{\min}\oplus \mathfrak{a}_{\min}$:
\[
\ImInfCh(\pi) = [(0,\nu)]\in \bigl ( \mathfrak{t}_{\min}^* \oplus i\mathfrak{a}_{\min}^*\bigr ) /  W(\mathfrak{g}_{\C}, \mathfrak{h}_{\min,\C}).
\]
The $W(\mathfrak{g}, \mathfrak{a}_{\min})$-orbit of $\nu$ in $i \mathfrak{a}_{\min}^*$ is uniquely determined by $\ImInfCh(\pi)$.
\end{theorem}

From Lemma~\ref{lem-helgason}, the inclusion of $i \mathfrak{a}_{\min}^*$ into $ \Im(\mathfrak{h}_{\min,\C})^*$ determines a continuous map 
\[ 
i \mathfrak{a}_{\min}^*/W(\mathfrak{g}, \mathfrak{a}_{\min}) 
\longrightarrow 
\Im(\mathfrak{h}_{\min,\C})^* / W(\mathfrak{g}_\C,\mathfrak{h}_{\min,\C})
\]
that is a homeomorphism onto its image.  From Theorem~\ref{thm-vogan-and-maybe-knapp} there is a diagram of  continuous maps 
\begin{equation}
    \label{eq-im-inf-ch-map}
\xymatrix@C=50pt{
 \widehat G_{\mathrm{tempered}}\ar@{=}[d] \ar[r]^-{\ImInfCh }&  \Im(\mathfrak{h}_{\min,\C})^* / W(\mathfrak{g}_\C,\mathfrak{h}_{\min,\C}) 
 \\
  \widehat G_{\mathrm{tempered}} \ar[r]_-{\ImInfCh } & i \mathfrak{a}_{\min}^*/W(\mathfrak{g}, \mathfrak{a}_{\min}) .\ar[u]
}
\end{equation}
We have labeled the bottom map, like the top map, as the imaginary part of the the infinitesimal character, and we shall use the bottom version of the imaginary part of the the infinitesimal character from now on.

\subsection{Stratification of the space of imaginary values of the infinitesimal character}

Recall that we fixed a system \eqref{eq-fixed-system-of-positive-roots} of positive  restricted roots 
and that we are denoting by 
$
S  \subseteq \Delta^{+} (\mathfrak{g} ,\mathfrak{a}_{\min})$ 
the subset of positive (reduced) roots. 

\begin{definition} We shall write 
\[
\mathfrak{a}_{\dom}^* = \{\, \nu\in \mathfrak{a}_{\min}^* : \langle \nu, \alpha \rangle \ge 0\,\, \forall \alpha \in S\,\} .
\]
This is the dominant Weyl chamber in $\mathfrak{a}_{\min}^*$ that is associated to the given choice of system of positive roots.

\end{definition}
The natural map
\[
\mathfrak{a}_{\dom}^*\stackrel \cong \longrightarrow  \mathfrak{a}_{\min}^* / W(\mathfrak{g},\mathfrak{a}_{\min})
\]
induced from the inclusion of $\mathfrak{a}_{\dom}^*$ into $\mathfrak{a}_{\min}^*$ is a homeomorphism, and from now on we shall use this homeomorphism to identify imaginary parts of infinitesimal characters with points in $\mathfrak{a}_{\dom}^*$.

\begin{definition}
\label{def-a-i-plus}
For every $I \subseteq S$ define 
\[
\mathfrak{a}^*_{I,+} 
     = \bigl\{ \, \nu\in \mathfrak{a}_{\dom}^* : \langle \nu,\alpha\rangle = 0 \,\, \forall \alpha \in I
\,\, \& \,\,\langle \nu,\beta\rangle \ne 0 \,\, \forall \beta \notin I \,\bigr\} .
\]
\end{definition} 

Note that $\lie{a}_{I,+}^* \subseteq \lie{a}_I^*$ as in \eqref{eq-characterization-of-a-i-star}.  Explicitly, we have
\begin{equation}
    \label{eq-fmla-for-a-i-plus-star}
 \lie{a}_{I,+}^* = \mathfrak{a}^*_{\dom}\cap \lie{a}_I^* \setminus \bigcup_{I\subsetneqq J} \lie{a}_J^*.
\end{equation}
The sets $\lie{a}_{I,+}^*$  are locally closed subsets of $\mathfrak{a}_{\dom}^*$, and   
the full chamber  $\mathfrak{a}_{\dom}^*$ is the disjoint union of all of them:
\begin{equation}
    \label{eq-partition-of-space-of-im-inf-chs}
\mathfrak{a}_{\dom}^* = \bigsqcup _{I \subseteq S} \mathfrak{a}^*_{I,+}.
\end{equation}
At one extreme, when $I=\emptyset$, $\mathfrak{a}_{I,+}^*$  is the interior of the dominant chamber; at the other, when  $I{=}S$, $\mathfrak{a}_{I,+}^*$  is  the   intersection of $\mathfrak{a}_{\min}$ with the center of $\mathfrak{g} $ (and so for instance it is a point if the center of $G$ is compact).  

\subsection{Classification of representations via the imaginary part of the infinitesimal character}

In this section we shall describe Vogan's characterization of the fibers of the map $\ImInfCh$ in \eqref{eq-im-inf-ch-map} over the various components in \eqref{eq-partition-of-space-of-im-inf-chs}.

\begin{definition}
\label{def-real-infinitesimal-character}
Let $G$ be a real reductive group. 
\begin{enumerate}[\rm (i)]
\item 
An irreducible unitary representation  $\pi$ of a real reductive group is said to have \emph{real infinitesimal character} if
\[
\ImInfCh (\pi) =0\in \mathfrak{a}_{\dom}^* .
\]

\item An irreducible, tempered unitary representation of $G$ is said to be  \emph{tempiric} (terminology of Afgoustidis) if it has  real infinitesimal character.  We shall denote by $\widehat G_{\mathrm{tempiric}}$ the set of unitary equivalence classes of  tempiric representations of $G$.

\end{enumerate}
\end{definition}

\begin{theorem}[{\cite[Cor.~3.5(2)]{Vogan00}}]
\label{thm-vogan-characterization-of-unitary-reps-with-fixed-im-inf-ch}
Let $I$ be a subset of the set of simple roots $S\subseteq \Delta(\mathfrak{g},\mathfrak{a}_{\min})$,  let  $\nu\in  \mathfrak{a}_{I,+}^*$, and denote by $P_I$ the associated standard parabolic subgroup of $G$, as in \eqref{eq-fixed-minimal-parabolic-subgroup}. The correspondence 
\[
\sigma \longmapsto \Ind_{P_I}^G \sigma {\otimes} \exp(i \nu) 
\]
determines a bijection from the set of unitary  equivalence classes of tempiric representations $\sigma$ of $M_I$  to the set of unitary  equivalence classes of irreducible tempered unitary representations $\pi$ of $G$ for which 
$\ImInfCh(\pi) = \nu$.
\end{theorem}

\begin{remark}
Actually, Vogan  considered general unitary representations in \cite{Vogan00}, rather than tempered representations, and for these the result is much more substantial.  Vogan credits the characterization  in this case to Knapp \cite[Thm.~16.10]{KnappRepTheorySemisimpleGroups}.  The proof of the tempered case that is stated above as Theorem~\ref{thm-vogan-characterization-of-unitary-reps-with-fixed-im-inf-ch} is relatively straightforward; see Section~\ref{sec-proof-of-classification}. 
\end{remark}

Theorem~\ref{thm-vogan-characterization-of-unitary-reps-with-fixed-im-inf-ch} determines a bijection
\begin{equation}
    \label{eq-description-of-tempered-dual-via-im-inf-ch}
    \widehat G_{\mathrm{tempered}} \cong  \bigsqcup _{I\subseteq S}
    \widehat M_{I,\mathrm{tempiric}} \times i \mathfrak{a}^*_{I,+}.
\end{equation}
This will be  central to what follows.

\section{A decomposition of the reduced group C*-algebra via the imaginary part of the infinitesimal character}
\label{sec-reduced-c-star-algebra-and-im-inf-ch}

The purpose of this section is to obtain from the decomposition of the tempered dual in \eqref{eq-description-of-tempered-dual-via-im-inf-ch} above a filtration of the reduced $C^*$-algebra of the real reductive group $G$ by ideals whose subquotients correspond to the individual parts in \eqref{eq-description-of-tempered-dual-via-im-inf-ch}, indexed by $I\subseteq S$.

\subsection{The reduced group C*-algebra and the infinitesimal character}

Let $G$ be a real reductive group and denote by $C^*_r (G)$ the reduced group $C^*$-algebra of $G$. In this section we shall use the theory of the imaginary part of the infinitesimal character that we reviewed in the last section to fiber the spectrum $\widehat{G}_{\temp}$ of $C_r^*(G)$ over $i\lie{a}_{\dom}^*$. In order to do this, we must prove that the imaginary part of the infinitesimal character defines a continuous map.

\begin{theorem}
\label{thm-im-inf-ch-for-the-reduced-c-star-algebra}
    The map
    \[\ImInfCh: \widehat{G}_{\temp} \to i\lie{a}^*_{\dom}\]
    is continuous.
\end{theorem}

\begin{remark} 
The full infinitesimal character is also continuous map, of course, see Theorem \ref{thm-continuity-of-inf-ch-for-g}, but it is the imaginary part which is essential to our view of the Connes-Kasparov isomorphism.
\end{remark}

\begin{definition}
\label{def-inf-ch}
    Let $G$ be any Lie group and let $T\in \mathcal{Z}(\mathfrak{g}_\C)$.  Define a function
    \begin{equation}
    \label{eq-continuity-of-inf-ch}
    \InfCh_T\colon \widehat{G} \longrightarrow \C
    \end{equation}
    on the unitary dual of $G$ by
    \[
    \InfCh_T(\pi)I = \pi(T)
    \]
    (this is an identity between operators acting on the smooth vectors in the Hilbert space of the representation $\pi$).
\end{definition}

We shall use the following simple and well-known fact.  We include a proof because we shall use the same argument to prove a less familiar fact later on, in Lemma~\ref{lem-continuity-of-inf-ch-in-families}. 

\begin{lemma}
\label{lem-continuity-of-inf-ch}
    Let $G$ be any Lie group.  For each $T \in \mathcal{Z}(\mathfrak{g}_\CC)$, the function $\InfCh_T$ is a continuous function on the unitary dual of $G$ \textup{(}and in particular, in the reductive case, on the tempered dual\textup{)}. 
\end{lemma}

In this paper we shall mostly understand the topology on the dual through  the bijection between closed sets in the spectrum of a $C^*$-algebra and closed, two-sided ideals in that $C^*$-algebra (in which  the closed set in the spectrum associated to an ideal is   the set of representations that vanish on the ideal). See \cite[Ch.3 and Ch.13]{DixmierEnglish} for information about the topology of the unitary dual, and for information about the topology on the spectra of $C^*$-algebras in general.  

But for the proof of the lemma, it will be better to use the following characterization of convergence in the unitary dual in terms of matrix coefficients: a net $\{ \pi_\alpha \}$  of irreducible unitary representations of $G$ converges to an irreducible unitary representation $\pi$ in the dual if and only if there are unit vectors $v_\alpha$ in the Hilbert spaces of the representations $\pi_\alpha$, and a unit vector $v$ in the Hilbert space of $\pi$, such that 
\begin{equation}
    \label{eq-topology-on-the-dual-from-matrix-coefficients}
\lim_{\alpha \to \infty} \langle v_\alpha, \pi_\alpha (a) v_\alpha \rangle = \langle v, \pi (a) v \rangle \qquad \forall a \in C_c^\infty (G).
\end{equation}
See again \cite[Ch.3 and Ch.13]{DixmierEnglish}.

\begin{proof}[Proof of Lemma~\ref{lem-continuity-of-inf-ch}] 
Let $T \in \mathcal{Z}(\mathfrak{g}_\C)$. Given \eqref{eq-topology-on-the-dual-from-matrix-coefficients} above, we wish to prove that 
\begin{equation}
    \label{eq-continuity-of-inf-ch-limit}
\lim_{\alpha \to \infty}\InfCh_T (\pi_\alpha) = \InfCh_T (\pi).
\end{equation}
Choose $a\in C_c^\infty (G)$ such that $\langle v, \pi(a) v \rangle \ne 0$ and then write 
\[
\InfCh_T (\pi_\alpha) \langle  v_\alpha, \pi_\alpha (a) v_\alpha\rangle =
 \langle  v_\alpha, \pi_\alpha (T)\pi_\alpha (a) v_\alpha\rangle
= \langle  v_\alpha, \pi_\alpha (Ta) v_\alpha\rangle,
\]
where in the last inner product we view $T$ as a differential operator on $G$. We find that 
\[
\lim_{\alpha \to \infty} \InfCh_T (\pi_\alpha) \langle  v_\alpha, \pi_\alpha (a) v_\alpha\rangle = 
\langle  v, \pi (Ta) v\rangle = \InfCh_T (\pi) \langle  v, \pi (a) v\rangle,
\]
from which \eqref{eq-continuity-of-inf-ch} follows.
\end{proof}

\begin{definition}
    Let $\mathfrak{h}$ be a Cartan subalgebra of $\mathfrak{g}$.
    Define the algebra morphism 
    \[
    \InfCh_{\lie{h}}: \widehat{G}_{\temp} \longrightarrow \lie{h}^*_{\CC}/W(\lie{g}_{\CC}, \lie{h}_{\CC})
    \]
    by means of the formula 
    \[
    \InfCh_{\lie{h}}(\pi) = [\lambda]\quad \text{when} \quad \InfCh(\pi) = \xi_{\lie{h},\lambda}.
    \]
    When $\mathfrak{h} = \mathfrak{h}_{\min}$, we shall often write simply $\InfCh$ in place of $\InfCh_{\mathfrak{h}_{\min}}$.
\end{definition}

\begin{theorem}\label{thm-continuity-of-inf-ch-for-g}
Let $G$ be a real reductive group and let $\lie{h}$ be a Cartan subalgebra of $\lie{g}$. 
The map
\[\InfCh_{\lie{h}}: \widehat{G}_{\temp} \to \lie{h}_{\C}^*/W(\lie{g}_{\C},\lie{h}_{\C})\]
is continuous.
\end{theorem}

\begin{proof}
It suffices to prove that the composition $f \circ \InfCh_{\lie{h}}$ is continuous for any $f \in \Sym(\lie{h}_{\C})^{W(\lie{g}_{\C},\lie{h}_{\C})}$. This follows from Lemma \ref{lem-continuity-of-inf-ch} and the Harish-Chandra homomorphism. \end{proof}

\begin{proof}[Proof of Theorem~\ref{thm-im-inf-ch-for-the-reduced-c-star-algebra}]
The inclusion
\[i\lie{a}^*_{\dom} \xrightarrow{\cong} i\lie{a}^*_{\min}/W(\lie{g},\lie{a}_{\min}) \to \Im(\lie{h}_{\min,\C})^*/ W(\lie{g}_\C,\lie{h}_{\min,\C})\]
is a homeomorphism onto its image. Therefore it suffices to prove that the composition of $\ImInfCh$ with the above inclusion is continuous. This composition is simply $\InfCh$ composed with the (continuous) projection
\[
 \lie{h}^*_{\min,\C}/W(\lie{g}_{\C},\lie{h}_{\min,\C})
 \longrightarrow 
 \Im (\lie{h}^*_{\min,\C})/W(\lie{g}_{\C},\lie{h}_{\min,\C}),
\]
so continuity follows from Theorem \ref{thm-continuity-of-inf-ch-for-g}.
\end{proof}

\begin{remark}\label{remark-embed-into-g-center-dauns-hofmann}
    By means of the Dauns-Hof\-mann theorem (see \cite{DaunsHoffman68} or \cite[Sec.\,4.4]{Pedersen79} for an exposition), the continuity of $\ImInfCh$ allows us to embed $C_0(i\lie{a}^*_\dom)$ into the center of the multiplier algebra of $C_r^*(G)$. 
    More precisely, there is a unique inclusion of the $C^*$-algebra $C_0(i \mathfrak{a}_{\dom}^*)$ into the center of the multiplier algebra of $C^*_r(G)$ such that 
    \[
    \pi(f\cdot a) = f( \ImInfCh (\pi)) \cdot \pi (a)
    \]
    for every tempered irreducible representation $\pi$, every $f\in C_0(i \mathfrak{a}_{\dom}^*)$ and every $a\in C^*_r (G)$.
\end{remark}

\subsection{Definition of the subquotients}
\label{subsec-filtration-of-c-star-g-by-ideals}
We shall use the imaginary part of the infinitesimal character to construct a (finite) filtration of $C^*_r(G)$ by ideals, whose subquotients have a particularly simple structure.

\begin{definition}
\label{def-ideals-and-quotients-of-the-reduced-c-star-algebra}
For every $I\subseteq S $ define 
\[
U_I 
    = \bigl\{ \, \nu\in \mathfrak{a}_{\dom}^* : 
\langle \nu,\beta\rangle \ne 0 \,\, \forall \beta \notin I \,\bigr\} 
     = \bigsqcup _{J\subseteq I} \mathfrak{a}^*_{J,+} .
\]
Thanks to the continuity of $\ImInfCh$, the tempered representations whose imaginary part of the infinitesimal character lies in $U_I$ forms an open subset of $\widehat{G}_{\temp}$. We shall denote by 
    \[
    C^*_r(G;U_I)\triangleleft C^*_r (G)
    \]
the $C^*$-algebra ideal in the reduced group $C^*$-algebra associated to this open subset of $\widehat{G}_{\temp}$.  That is, $C^*_r(G,U_I)$ is the mutual kernel of all the irreducible representations $\pi$ of $C^*_r(G)$ for which $\ImInfCh(\pi)$ lies in the complement of $U_I$.

It follows from the formula \eqref{eq-fmla-for-a-i-plus-star} that $\lie{a}_{I,+}^*$ is a closed subset of $U_I$.
We shall denote by 
    \[
    \pi_I \colon C^*_r(G;U_I)\longrightarrow C^*_r (G; I)
    \]
the quotient associated to this closed subset.  That is, $C^*_r (G; I)$ is the quotient of $C^*_r(G;U_I)$ by the ideal consisting of all elements that vanish in every representation $\pi$ with 
$\ImInfCh(\pi)\in U_I \setminus \mathfrak{a}^*_{I,+}$.
\end{definition}

By elementary $C^*$-algebra theory \cite[Ch.2]{DixmierEnglish}, every irreducible representation $\pi$ of $C^*_r (G)$ for which $\ImInfCh(\pi) \in U_I$ restricts to an irreducible   representation of $C^*_r (G ; U_I)$, and moreover every irreducible   representation of $C^*_r (G ; U_I)$ extends uniquely to an irreducible representation $\pi$ of $C^*_r (G)$ for which $\ImInfCh(\pi) \in U_I$.  Moreover, an irreducible   representation $\pi$ of $C^*_r (G ; U_I)$ factors through the quotient $C^*_r(G;I)$ if and only if $\ImInfCh(\pi)\in \mathfrak{a}^*_{\min,I}$.  More precisely: 

\begin{lemma} 
The procedure that associates to  each  irreducible representation $\pi'$ of   $C^*_r(G;I)$   the unique extension   to $C^*_r (G)$ of the composition  of $\pi'$ with the quotient map $C^*_r (G;U_I)\to C^*_r (G;I)$ in Definition~\textup{\ref{def-ideals-and-quotients-of-the-reduced-c-star-algebra}} 
determines  a   homeomorphism   
\[
 \pushQED{\qed}
\widehat{C^*_r(G;I)} \stackrel \cong \longrightarrow \bigl \{ \,[\pi]\in  \widehat G_{\mathrm{tempered}} : \ImInfCh(\pi) \in \mathfrak{a}^*_{I,+}\,\bigr \}.
\qedhere\popQED
\]
\end{lemma}

\subsection{Discrete series representations and tempiric representations}

We shall use the following substantial results of Harish-Chandra about the discrete series representations  of real reductive groups:

\begin{theorem}
\label{thm-uniform-admissibility-for-discrete-series}
    Let $G$ be a real reductive group and let $K$ be a maximal compact subgroup of $G$.  Every irreducible representation of $K$ occurs as a $K$-type in at most finitely many discrete series representations of $G$. 
\end{theorem}

\begin{remark}
    In view of the aims of this paper, it is interesting to note that this result follows from the Connes-Kasparov isomorphism. See Section~\ref{appendix-connes-kasparov}.
\end{remark}

Theorem~\ref{thm-uniform-admissibility-for-discrete-series} has the following consequence: 

\begin{theorem}
\label{thm-full-uniform-admissibility}
    Let $G$ be a real reductive group and let $K$ be a maximal compact subgroup of $G$.  Every irreducible representation of $K$ occurs as a $K$-type in at most finitely many  tempiric representations of $G$. 
\end{theorem}

\begin{proof}
By Theorem \ref{thm-harish-chandra-principle-on-cuspidal-principal-series} and \eqref{eq-inf-ch-of-principal-series}, for every tempiric representation $\pi$ of $G$ there is a parabolic subgroup $P {=} MAN$ of $G$  and a discrete series representation $(\sigma, H_\sigma)$ of $M$ such that $\pi$ embeds as a subrepresentation of $\Ind_P^G \sigma {\otimes} 1$. By Frobenius reciprocity, 
\[\Hom_K(V_\gamma, \Ind_P^G \sigma {\otimes} 1) \cong \Hom_{M \cap K}(V_\gamma, H_\sigma).\]
Therefore, $V_\gamma$ occurs in $\Ind_P^G \sigma {\otimes} 1$ only if an $(M{\cap} K)$-type occurring in $V_\gamma$ also occurs in $H_\sigma$. By Theorem \ref{thm-full-uniform-admissibility}, there are finitely many discrete series representations $\sigma$ of $M$ which contain one of the (finitely many) $(M{\cap} K)$-types occurring in $V_\gamma$. We conclude that $V_\gamma$ is contained in finitely many representations of the form $\Ind_P^G(\sigma \otimes 1)$, and in particular it is contained in finitely many tempiric representations.
\end{proof}

\begin{corollary}
\label{cor-discrete-topology-on-tempiric dual}
    The topological subspace  $\widehat G_{\mathrm{tempiric}}\subseteq \widehat G$ is a closed subspace, and the topology that it inherits from $\widehat G$ is the discrete topology \textup{(}in which every subset is an open set\textup{)}. \qed
\end{corollary}

\begin{proof}
    The subspace $\widehat{G}_{\mathrm{tempiric}} \subseteq \widehat{G}_{\mathrm{tempered}}$ is the preimage of $0$ under $\ImInfCh$, and is therefore closed.  Theorem \ref{thm-full-uniform-admissibility} shows that the topology of $\widehat{G}_{\mathrm{tempiric}}$ has a basis of finite open sets.  But points in the unitary dual of a real reductive group are closed in the topology of the unitary dual (see for instance \cite[Sec.~15.5]{DixmierEnglish}).  The result follows.
\end{proof}

\subsection{Structure of the subquotients}

\nocite{CCH16}

In this section we shall fix a subset $I{\subseteq} S$, and determine $C^*_r (G;I)$ up to isomorphism.

\begin{theorem}
\label{thm-surjective-morphism-from-the-i-subquotient-of-g}
For every $ \sigma\in \widehat M_{I,\mathrm{tempiric}}$ there is a unique surjective $C^*$-algebra morphism
    \[
    \pi_\sigma \colon C^*_r (G;I)\longrightarrow  C_0\bigl (\mathfrak{a}^*_{I,+}, \mathfrak{K} (\Ind_{P_I}^G H_\sigma)\bigr )
    \]
such that  
\[
\pi_\sigma(a)(\nu) =  \pi_{\sigma,\nu}(a)
\]
for every $\nu\in \mathfrak{a}^*_{I,+}$ and every $a\in C^*_r (G;I)$.
\end{theorem}

\begin{proof} 
Of course, we would like to \emph{define} $\pi_\sigma$ by the formula
\[
\pi_\sigma(a)(\nu) = \pi_{\sigma,\nu}(a)
\]
that appears in the statement of the theorem.
That we may do so follows from the considerations in \cite[Lemma~4.10]{CCH16}, which show that $\pi_{\sigma,\nu}(a)$ is a continuous function of $\nu$, and indeed a $C_0$-function of $\nu$ (this is essentially the Riemanna-Lebesgue lemma; compare the proof of Lemma~\ref{lem-mackey-embedding} below, which deals with the same issue for a different group).  It remains to prove that $\pi_\sigma$ is surjective; this follows from Dixmier's Stone-Weierstrass theorem \cite[Thm.~11.1.8]{DixmierEnglish}.
\end{proof} 

\begin{theorem}
\label{thm-i-subquotient-of-g-isomorphism}
There is a unique $C^*$-algebra isomorphism
    \[
     C^*_r (G;I)\stackrel \cong\longrightarrow \bigoplus _{\sigma\in \widehat M_{I,\mathrm{tempiric}}} C_0\bigl (\mathfrak{a}^*_{I,+}, \mathfrak{K} (\Ind_{P_I}^G H_\sigma)\bigr )
    \]
whose composition with the projection onto the $\sigma$-summand in the direct sum is $\pi_\sigma$.
\end{theorem}

\begin{proof}
    This follows from Theorem~\ref{thm-surjective-morphism-from-the-i-subquotient-of-g} and the fact that the supports of the representations $\pi_\sigma$, for $\sigma\in \widehat{M}_{I,\mathrm{tempiric}}$, are pairwise disjoint.  Compare \cite[Lemma~5.14]{CCH16}.
\end{proof}

\section{The Cartan motion group}
\label{sec-cartan-motion-group}

Let $G$ be a real reductive group and let $K$ be a  maximal compact subgroup.  The associated \emph{Cartan motion group} is the semidirect product group 
\begin{equation*}
G_0 = K \ltimes (\mathfrak{g}/\mathfrak{k})
\end{equation*}
that is constructed from the adjoint action of $K$ on $\mathfrak{g}/\mathfrak{k}$. In this section we shall review the unitary representation theory of $G_0$ and related matters, and construct a theory of the imaginary part of the infinitesimal character for the irreducible unitary representations of $G_0$.

\subsection{Irreducible unitary representations of the Cartan motion group}

In what follows we shall use the Cartan decomposition 
$
\mathfrak{g} = \mathfrak{k} \oplus \mathfrak{s}$
to  identify the quotient vector space  $\mathfrak{g}/\mathfrak{k}$ with $\mathfrak{s}$, and so write the Cartan motion group   as
\begin{equation}
\label{eq-def-of-cartan-motion-group-with-s}
G_0 = K\ltimes \mathfrak{s}.
\end{equation}
Let  $\xi$ be a  linear functional on the real vector space $\mathfrak{s}$, and let $K_\xi$ be its isotropy subgroup   in $K$ for the coadjoint action. Mackey \cite{Mackey49} proved that if $\tau$ is any irreducible unitary representation of $K_\xi$, then the induced unitary representation 
\begin{equation}
\label{eq-def-rho-tau-xi}
\rho_{\tau,\xi}= \Ind_{K_\xi\ltimes \mathfrak{s}}^{K\ltimes \mathfrak{s}} \tau\otimes \exp (i \xi) 
\end{equation}
is an irreducible unitary representation of $G_0$.  He further proved that every irreducible unitary representation of $G_0$ is equivalent to a representation of this form, and that the only equivalences among these representations are the obvious ones coming from $K$-conjugacies, so that from \eqref{eq-def-rho-tau-xi} we obtain a bijection
\begin{equation}
\label{eq-mackey-formula-for-motion-group-dual}
\widehat G_0 \cong  \Bigr ( \bigsqcup_{\xi\in \mathfrak{s}^*}   \widehat K_\xi \Bigl ) \Big / K .
\end{equation}
For a review of these results from  the $C^*$-algebraic point of view of this paper, see \cite[Sec.\,3.1]{Higson08}.

The bijection in \eqref{eq-mackey-formula-for-motion-group-dual} can  adapted to our current purposes by using the inclusion of $\mathfrak{a}_{\min}$ into $\mathfrak{s}$, along with  the following terminology and conventions (for our later purposes we shall give  more  details  in the following definition than are immediately needed).

\begin{definition}
    \label{def-rho-of-sigma-nu}
Let $I$ be a subset of the set of simple roots in $\Delta ^+ (\mathfrak{g},\mathfrak{a}_{\min})$ and let  $\nu\in \mathfrak{a}_{I,+}^*$ (notation from Definition~\ref{def-a-i-plus}).  Let $K_I$ be the centralizer of $\mathfrak{a}_I$ in $K$ (this may  also be described as the isotropy group of the single element $\nu$, as well as the intersection $K{\cap}M_I$) and let $\sigma$ be a unitary representation of $K_I$.
We shall denote by $\rho_{\sigma,\nu}$ the unitary representation 
\[
\rho_{\sigma, \nu} = \Ind_{K_I\ltimes \mathfrak{s}}^{K\ltimes \mathfrak{s}} \sigma\otimes \exp (i \nu) 
\]
of $G_0 = K\ltimes \mathfrak{s}$. Here $\nu$ is extended to a linear functional on $\mathfrak{s}$ that is zero on the orthogonal complement of $\mathfrak{a}_I$ in $\mathfrak{s}$.  We shall realize the representation as follows: the Hilbert space of the representation is $\Ind_{K_I}^K H_\sigma  = L^2(K,H_\sigma)^{K_I}$ and the action of $G_0$ on an element $f \in L^2(K,H_\sigma)^{K_I}$ is determined by
\[
\begin{cases} 
(k\cdot f )(k_1) = f (k^{-1}k_1) & \forall k,k_1\in K,
\\
(X \cdot f)(k_1) = e^{i \nu (\Ad_{k_1}^{-1}(X))}f(k_1) &  \forall X \in \mathfrak{s}\,\,\forall k_1\in K.
\end{cases}
\]
\end{definition}

\begin{remark}
For the time being we are interested in the cases of the definition where $\sigma$ is an irreducible unitary representation of $K_I$, but later we shall consider some unitary but reducible representations, too. 
\end{remark}

With this notation established, the bijection \eqref{eq-mackey-formula-for-motion-group-dual} can be recast as follows: 

\begin{theorem}
     The representations $\rho_{\tau, \nu}$ of $G_0$, as $I$ ranges over the finite subsets of the set of $S$ of simple restricted roots, as $\tau $ ranges over the unitary dual of $K_I$, and as $\nu$ ranges over $\mathfrak{a}^*_{I,+}$, are irreducible and pairwise inequivalent.  They exhaust the unitary dual of $G_0$, and determine a bijection
\begin{equation*}
  \pushQED{\qed} \widehat G_0 \cong \bigsqcup _{I\subseteq S}  \widehat K_I \times \mathfrak{a}^*_{I,+}. 
  \qedhere \popQED
\end{equation*}
\end{theorem}

\begin{remark}
Of course, this is to be compared with the similar bijection for the tempered dual of $G$ in \eqref{eq-description-of-tempered-dual-via-im-inf-ch}.  The two bijections are the foundation of Afgoustidis's Mackey bijection \cite{AfgoustidisMackeyBijection}.
\end{remark}

\subsection{Imaginary part of the infinitesimal character for the motion group} 
 
The vector subspace $\mathfrak{s}\subseteq\mathfrak{g}$, considered as an abelian Lie algebra, is a Lie subalgebra of the Lie algebra of the Cartan motion group $G_0$, and there is an inclusion
\[
    \operatorname{Sym}(\mathfrak{s})^K\longrightarrow \mathcal{Z}(\mathfrak{g}_0)^K .
\]
Often, but not always, this is an isomorphism.  In any case, each $T\in \operatorname{Sym}(\mathfrak{s})^K$ determines a continuous map 
\begin{equation}
    \label{eq-def-inf-ch-sub-t-for-motion-group}
\InfCh_T\colon \widehat G_0 \longrightarrow \C
\end{equation}
by the formula 
$
\InfCh_T(\pi) \cdot I = \pi (T)$ 
(action on smooth vectors).  

The map \eqref{eq-def-inf-ch-sub-t-for-motion-group} is easy to compute: identifying the element $T\in \operatorname{Sym}(\mathfrak{s})^K$ with a $K$-invariant polynomial function on $\mathfrak{s}^*$ we have 
\[
\InfCh_T(\rho_{\tau,\nu}) = T(\nu).
\]
\begin{lemma}\label{lem-continuity-inf-ch-for-g-0}
\begin{equation}
\label{eq-def-inf-ch-on-unitary-dual-for-motion-group}
\begin{gathered}
\InfCh  \colon \widehat G_0 \longrightarrow i \mathfrak{s}^*/K
\\
\InfCh (\rho_{\tau, \nu})   = i\nu \in i\mathfrak{s}^* /K,
\end{gathered} 
\end{equation}
is continuous.
\end{lemma}
\begin{proof}
This follows from Lemma~\ref{lem-continuity-of-inf-ch}. 
\end{proof}

\begin{remark} 
As in Remark \ref{remark-embed-into-g-center-dauns-hofmann}, the Dauns-Hofmann theorem allows us to embed $C_0(i\lie{s}^*/K)$ into the center of the multiplier algebra of $C^*(G_0)$. However, the Dauns-Hofmann theorem is not necessary as Fourier theory \cite[Thm.\,3.2]{Higson08} provides an isomorphism
\[
C^*_r (G_0)\cong C\bigl (i\mathfrak{s}^*, \mathfrak{K}(L^2(K))\bigr )^K
\]
under which the action of  $C_0(i\mathfrak{s}^*/K)$  is simply pointwise multiplication.
\end{remark}

Since the inclusion of $i \mathfrak{a}_{\min}^*$ into $i \mathfrak{s}^*$  as imaginary-valued linear functionals vanishing on the orthogonal complement of $\mathfrak{a}_{\min}$ induces a homeomorphism 
\[
i \mathfrak{a}^*_{\min} / W(\mathfrak{g},\mathfrak{a}_{\min}) \stackrel \cong \longrightarrow i \mathfrak{s}^*/ K ,
\]
we can view $\InfCh(\pi)$ as an element of $ i\mathfrak{a}_{\min}^*/ W(\mathfrak{g},\mathfrak{a}_{\min})$, and so obtain from \eqref{eq-def-inf-ch-on-unitary-dual-for-motion-group} a continuous map
\begin{equation}
    \label{eq-def-inf-ch-on-unitary-dual-for-motion-group-2}
\InfCh  \colon \widehat G_0 \longrightarrow i\mathfrak{a}_{\min}^*/ W(\mathfrak{g},\mathfrak{a}_{\min}).
\end{equation}
A comparison between this and the map $\ImInfCh$ that we defined in the reductive case, see \eqref{eq-im-inf-ch-map}, makes it reasonable to also use the notation $\ImInfCh$ for the map in \eqref{eq-def-inf-ch-on-unitary-dual-for-motion-group-2}, which we shall do from now on.  In other words, we are effectively declaring the real part of the infinitesimal character to be always zero on $\widehat{G}_0$. Thus 
\begin{equation}
    \label{eq-def-im-inf-ch-on-unitary-dual-for-motion-group}
\ImInfCh(\rho_{\tau,\nu}):=i\nu \in i\mathfrak{a}_{\min}^*/ W(\mathfrak{g},\mathfrak{a}_{\min}) \cong i\mathfrak{a}^*_{\dom}.
\end{equation}

\subsection{Subquotients of the motion  group C*-algebra}

Using the continuous map $\ImInfCh$ just defined on $\widehat{G}_0$, we can now repeat the constructions of Section~\ref{subsec-filtration-of-c-star-g-by-ideals} more or less word for word, with $G$ replaced by $G_0$.

\begin{definition}
\label{def-ideals-and-quotients-of-the-reduced-c-star-algebra-g-0}
We shall denote by 
    \[
    C^*(G_0;U_I)\triangleleft C^* (G_0)
    \]
the $C^*$-algebra ideal in $C^* (G_0)$ associated to the open subset of the unitary dual consisting of representations whose infinitesimal character belongs to the open subset $U_I\subseteq \mathfrak{a}_{\dom}^*$  in Definition~\ref{def-ideals-and-quotients-of-the-reduced-c-star-algebra}. We shall denote by $C^* (G_0; I)$ the quotient $C^*$-algebra associated to the closed subset $\mathfrak{a}^*_{\min,I}\subseteq U_I$, and by 
    \[
    \rho_I\colon C^*(G_0;U_I)\longrightarrow C^* (G_0; I)
    \]
the quotient mapping.
\end{definition}

\begin{theorem}
\label{thm-structure-of-g-0-component-algebra-prelim}
For every $ \tau\in \widehat K_I$ there is a unique surjective $C^*$-algebra morphism
    \[
    \rho_\tau \colon C^* (G_0;I)\longrightarrow  C_0\bigl (\mathfrak{a}^*_{I,+}, \mathfrak{K} (\Ind_{K_I}^K H_{\tau})\bigr )
    \]
such that   
\[
\rho_\tau(\rho_I(a))(\nu) =  \rho_{\tau,\nu}(a)
\]
for every $\nu\in \mathfrak{a}^*_{I,+}$ and every $a\in C^* (G_0;U_I)$.
\qed
\end{theorem}

\begin{theorem}
\label{thm-structure-of-g-0-component-algebra}
There is a unique $C^*$-algebra isomorphism
    \[
     C^* (G_0;I)\stackrel \cong\longrightarrow \bigoplus _{\tau\in \widehat K_{I}} C_0\bigl (\mathfrak{a}^*_{I,+}, \mathfrak{K} (\Ind_{K_I}^K H_\tau)\bigr )
    \]
whose composition with the projection onto the $\tau$-summand in the direct sum is the morphism $\rho_\tau$ from Theorem~\textup{\ref{thm-structure-of-g-0-component-algebra-prelim}}. \qed
\end{theorem}

\section{Scaling the imaginary part of the infinitesimal character}
\label{sec-deformation-family}

The Mackey deformation family \cite{Higson08, AfgoustidisConnesKasparov} is a smooth one-parameter family of Lie groups that  links a real reductive group $G$ to its Cartan motion group. 
In this section we shall extend the constructions of the previous sections---the imaginary part of the infinitesimal character and the associated filtration of the reduced group $C^*$-algebra by ideals---to the context of the Mackey family.

\subsection{Deformation families of Lie groups}
\label{subsec-deformation-family}

Let $G$ be a Lie group and let  $K$ be a compact subgroup of $G$ (or, for nearly everything that follows, any closed subgroup).  The deformation to the normal cone construction, applied to $K\subseteq G$, yields a smooth family of Lie groups  $  \{G_t\}_{t\in[0,\infty)}$ with
\begin{equation}
\label{eq:G_t}
  G_t = 
  \begin{cases}
    G & t > 0 \\
    K \ltimes (\lie{g/k}) & t=0.
  \end{cases}
\end{equation}
See \cite[Sec.\,6.2]{Higson08} for this particular instance of the deformation to the normal cone construction. 
We shall use the following features:  

\begin{enumerate}[\rm (i)]

\item   The disjoint union $\GG= \sqcup \{ G_t : t \geq 0\}$ is equipped with a smooth manifold (with boundary) structure for which the  map 
\[
  \GG  \to [0, \infty) ,\qquad G_t\ni g\mapsto t\in [0, \infty)
\]
 is a submersion.

    \item The fiberwise multiplication, inverse and unit  maps $\GG\times_{[0, \infty)} \GG\to \GG$,  $\GG\to\GG$ and $[0, \infty)\to \GG$ are  smooth maps.
    \item 
    If $V$ is a finite-dimensional real vector space, and if $E\colon V\to G$ is a smooth map for which $E(0)=e$, and for which  the map
    \[
    K \times V \to G,\qquad  (k,v) \mapsto k\cdot  E(v)
    \]
    is a diffeomorphism, then the map
    \[
      K \times V \times [0, \infty) \to \GG; \qquad 
      (k,v,t) \mapsto 
      \begin{cases}
          k\cdot  E(tv)\in G_t & t > 0 \\
          (k,  dE(v)) \in G_0 & t=0
      \end{cases}
    \]
    is a diffeomorphism, where $dE\colon V\to \mathfrak{g}/\mathfrak{k}$, is the composition of the derivative of $E$ at $0\in V$ with the projection  $\mathfrak{g}\to \mathfrak{g}/\mathfrak{k}$ (for this, we use $G_0 = K \ltimes \mathfrak{g} / \mathfrak{k}$).
\end{enumerate}

\begin{remark}
A diffeomorphism as in point (iii) exists in all of the cases of interest to us.  Specifically, when $G$ is a real reductive group and $K$ a maximal compact subgroup, we have the Cartan decomposition.  Likewise, when $P$ is a parabolic subgroup and $K$ is a maximal parabolic of $P$, we can combine the Langlands decomposition of the Levi factor with the exponential map on the nilpotent radical.  
\end{remark}

Equip the groups $G_t$ with Haar measures $\mu_t$, as follows.  Fix a Haar measure $\mu{=}\mu_1$ on $G{=} G_t$, and then define 
\begin{equation}
\label{eq-rescaled-Haar-measure}
    \mu_t {=} t^{-d} \mu \qquad \text{for }t > 0,
\end{equation} where $d {=} \dim (G/K)$.  As for $\mu_0$, the measure $\mu$ with which we started determines a Lebesgue measure on $\mathfrak{g}$, which may be written as the product of a Lebesgue measure $\mu_{\mathfrak{k}}$ on $\mathfrak{k}$ and a Lebesgue measure $\mu_{\mathfrak{s}}$ on $\mathfrak{s}$ (the summands in the Cartan decomposition $\mathfrak{g} = \mathfrak{k}\oplus \mathfrak{s}$).  The measure $\mu_{\mathfrak{k}}$ determines a Haar measure $\mu_K$ on $K$, and we define $\mu_0$ to be the product of $\mu_K$ and $\mu_{\mathfrak{s}}$ on $G \cong K\ltimes \mathfrak{s}$.

The measures $\mu_t$ vary smoothly with $t$ in the sense that if $f$ is a smooth and compactly supported function on $\GG$, then $\int_{G_t} f(g_t) \, d \mu_t g_t$ is a smooth function of $t$.  Moreover the operations of fiberwise convolution and adjoint  define   $*$-algebra structure on the space $C_c^\infty (\GG)$ of all smooth and compactly supported functions on $\GG$.

\begin{definition}
\label{def-family-group-c-star-algebra}
    We shall denote by $C^*_r(\GG)$ the $C^*$-algebra completion of the $*$-algebra $C_c^\infty (\GG)$ in the norm 
    \[
    \| f\| = \sup \bigl \{ \, \| f_t\|_{C^*_r (G_t)} : t\geq 0 \,\bigr \} ,
    \]
    where $f_t$ is the restriction of $f$ to $G_t\subseteq \GG$.
\end{definition}

\begin{remark}
\label{rem-continuous-field-from-DNC}
    It is shown in \cite[Lemma~6.13]{Higson08} that for $f\in C_c^\infty (\GG)$ the norms $\| f_t\|_{C^*_r (G_t)}$ vary continuously with $t$, so that we obtain from the constructions above a continuous field of $C^*$-algebras $\{ C^*_r (G_t)\} _{t \geq 0}$, in the sense of \cite[Ch.~10]{DixmierEnglish}.
\end{remark}

The spectrum of $C^*_r (\GG)$  is the disjoint union
\[
 \widehat{\GG}_{\mathrm{tempered}} : =  \widehat{C^*_r(\GG)} = \bigsqcup_{t\geq 0} \widehat{G}_{t,\temp} .
\]
We shall equip it with the usual hull-kernel topology it carries as the spectrum of a $C^*$-algebra.
But for  the most part, we shall only need the following simpler notion of convergence among representations of the groups $G_t$.

\begin{definition}
\label{def:family_of_reps}
  By a  \emph{continuous family of unitary representations} of $\GG$ on a fixed Hilbert space $H$ we shall mean a strongly continuous map $\pi:\GG\to U(H)$ whose restriction $\pi_t$ to each $G_t$ is a unitary representation.
\end{definition}

\begin{lemma}
If $\{\pi_t\}$ is a continuous family of representations in the above sense, and if $\pi_t$ is irreducible for all $t >0$, then every irreducible component of $\pi_0$ is a limit   of $\{\pi_t\}$ in $\widehat{\GG}$ as $t\to 0$. \qed   
\end{lemma}

\begin{example}
\label{ex:restriction_family}
Let $\pi:G \to U(H)$ be a unitary representation.  Define
\[
\begin{cases}
  \pi_t (g)= \pi(g) & t {>} 0 \\
  \pi_0 (k,X) = \pi(k) & t {=}0.
\end{cases}
\]
Then $\{\pi_t\}_{t\in[0,\infty)}$ is a continuous family of representations.
\end{example}

\begin{example}
\label{ex:dilation_family}
Let $P{=}MAN$ be a parabolic subgroup of a real reductive group. Form the deformation to the normal cone $\pmb P$ associated to $P$ and the compact subgroup $K{\cap}M$. Let $\sigma$ be an irreducible unitary representation of $M$ and let $\nu \colon \mathfrak{a} \to \R$ be a linear functional.     Define unitary representations $\lambda_t $ of $P_t$  on the Hilbert space of $\sigma$ by 
\begin{equation*}
  \begin{cases}
     \lambda_t (man) =  e^{i t^{-1} \nu (\log (a))}\cdot  \sigma(m) & t{ > } 0
  \\ 
  \lambda_0(k,X) = e^{i \nu(X)} \cdot \sigma(k) & t {=} 0,
  \end{cases} 
\end{equation*}
where $\nu$ is transferred to a linear functional on  $ \mathfrak{p}/(\mathfrak{k}{\cap} \mathfrak{m})$ by setting it equal to  zero on  the image of $\mathfrak{m}\oplus \mathfrak{n}$.   This  is a continuous family of representations of $\pmb P$.
\end{example}

The following proposition is a special case of a more general principle that unitary  induction takes continuous families to continuous families.

\begin{proposition}
\label{prop:continuous_induction}
Let $P{=}MAN$ be a parabolic subgroup of a real reductive group $G$.  Let $\sigma$ be a unitary  representation of $M$ and let $\nu \in \mathfrak{a}^*$.  The representations 
\[
\pi_t = \begin{cases} \pi_{\sigma, t^{-1} \nu} & t{> }0 \\
\rho_{\sigma,\nu} & t {=}0
\end{cases}
\]
form a continuous family of unitary representations of $\GG$.
\end{proposition}

\begin{proof}
Note that $\pmb{P}$ embeds as a smooth closed submanifold of $\GG$. For $g\in \GG$ we have a (non-unique) decomposition
\[
g= \kappa(g)\cdot \varphi(g)
\]
with $\kappa(g)\in K$ and $ \varphi(g)\in \pmb P$.   It is not usually possible to choose $\kappa$ and $\varphi$ to be globally smooth functions, but we can at least choose them to be smooth in a neighbourhood of any point $g\in\GG$. 

Now for  $\xi\in L^2 (K,H_\sigma)^{M\cap K}$ we may write 
\[
  (\pi(g)\xi)(k')
   = \delta^{-1/2}(\varphi (g^{-1}k'))  \cdot \lambda(\varphi(g^{-1}k'))^{-1} \xi(\kappa(g^{-1}k')),
\]
 where $\delta$ is the  modular function for $P$ (a necessary ingredient in unitary induction) and where $\lambda$ is as in Example~\ref{ex:dilation_family}.  This, plus the formula above,  show that $g\mapsto \pi(g)\xi$ is continuous from $g$ to $L^2(K,H_\sigma)$, as claimed.  
\end{proof}

\subsection{The imaginary part of the infinitesimal character for the Mackey deformation}

We turn now to infinitesimal characters in the context of the deformation to the normal cone family.  The fundamental observation underlying all of this is the following Lemma.  For the groups of interest to us, it is a consequence of the condition (iii) on $\GG$ in Section~\ref{subsec-deformation-family}.

\begin{lemma} 
Let $G$ be any Lie group, let $K$ be a compact subgroup of $G$ \textup{(}or indeed any closed subgroup\textup{)}, and let $X$ be a right-invariant vector field on $G$. The family of vector fields $X_t = tX$ on $G_t$, for $t{>} 0$, extends to a smooth family of vector fields on the fibers $G_t$ of the deformation family $\GG$.   If we consider $X$ as an element in $\mathfrak{g}$, then the  vector field $X_0$ on $G_0$ is the right-invariant vector field associated to the image of $X$ in $\mathfrak{g}/\mathfrak{k}$. \qed
\end{lemma}

Let us apply this to elements in the  enveloping algebra of $G$.   If $T\in \mathcal{U}(\mathfrak{g}_{\C})$, and if $T$ has order $k$ for the usual increasing filtration on the enveloping algebra (so that $T$ is a sum of products of $k$ or fewer elements from $\mathfrak{g}$), then the family of right-invariant differential operators 
\[
T_t = t^k T \colon C_c^\infty (G_t) \longrightarrow C_c^\infty(G_t)\qquad (t>0)
\]
extends to a smooth family of right-invariant differential operators defined on all of the fibers of $\GG$.  We obtain from this a families version of Lemma~\ref{lem-continuity-of-inf-ch}; we shall formulate it for real reductive groups and the tempered dual only because we have not defined the full unitary dual of $\GG$.  

\begin{lemma} 
\label{lem-continuity-of-inf-ch-in-families}
Let $G$ be a real reductive group and let  $T\in \mathcal{Z}(\mathfrak{g}_{\C})$. If $T$ has order $k$, and if $T_t = t^k T$ for $T{\ne }0$,  then the function 
 \[
 \widehat G_t \ni \pi_t \longmapsto \InfCh_{T_t}(\pi_t)\in \C\qquad (t>0)
 \]
 extends at $t{=}0$ to a continuous function on $\widehat {\GG}_{\mathrm{tempered}}$.
\end{lemma} 

\begin{proof} 
The  argument that was used for the proof of Lemma~\ref{lem-continuity-of-inf-ch} may be applied here, too.
\end{proof}

\begin{theorem}
\label{thm-continuous-inf-ch-on-family-dual}
    Let $G$ be a real reductive group and let $\mathfrak{h}_{\C}$ be a Cartan subalgebra of $\mathfrak{g}_{\C}$.   The family of maps 
    \[
   t \cdot \InfCh\colon  \widehat G _{t,\mathrm{tempered}} \longrightarrow \mathfrak{h}_{\C}^* / W(\mathfrak{g}_{\C},\mathfrak{h}_{\C})\qquad (t> 0)
    \]
    extends to a continuous map
    \[
    \widehat{\GG}_{\mathrm{tempered}} \longrightarrow \mathfrak{h}_{\C}^* / W(\mathfrak{g}_{\C},\mathfrak{h}_{\C}).
    \]
    Specifically, taking $\lie{h}=\lie{h}_{\mathrm{min}}$, the continuous extension to $t=0$ is given by the infinitesimal character of $G_0$,  from \eqref{eq-def-im-inf-ch-on-unitary-dual-for-motion-group},
    \[
        \InfCh:\widehat{G}_0 \to i\mathfrak{a}^*_{\mathrm{min}}/W(\lie{g},\lie{a}_{\mathrm{min}}) \subseteq \mathfrak{h}_{\min,\C}^* / W(\mathfrak{g}_{\C},\mathfrak{h}_{\C}).
    \]
\end{theorem}

\begin{proof} 
According to Chevalley's theorem \cite{Chevalley55}, there are  homogeneous, Weyl group-invariant  polynomial  functions  $p_1,\ldots, p_r $  on $\mathfrak{h}^*_{\C}$, say of degrees $k_1,\dots, k_r$, that freely generate the algebra of all  Weyl group-invariant  polynomial  functions on $\mathfrak{h}^*_{\C}$.  The function 
\[
\mathfrak{h}^*_{\C}/W(\mathfrak{g}_{\C},\mathfrak{h}_{\C})  \longrightarrow \C^r, \qquad \lambda\longmapsto (p_1(\lambda), \dots, p_r (\lambda))
\]
is then a homeomorphism (indeed an isomorphism of affine varieties).  The compositions of this homeomorphism with the maps in the statement of the theorem are the maps 
\[
 \widehat G_t \ni \pi_t \longmapsto \bigl ( \InfCh_{T_{1,t}}(\pi_t), \dots,  \InfCh_{T_{r,t}}(\pi_t)\bigr ),
\]
where where $T_j \in \mathcal{Z}(\mathfrak{g}_{\C})$ corresponds to $p_j\in \C[\mathfrak{h}_{\C}^*]$ under the Harish-Chandra isomorphism, and where $T_{j,t}= t^{k_j} T_j$. Since $T_j$ has order $k_j$, it follows from Lemma~\ref{lem-continuity-of-inf-ch-in-families} that the composition gives a continuous map from $\widehat{\GG}_{\mathrm{tempered}}$ to $\C^r$.  This proves the existence of the continuous extension. 

In order to determine the value of this extension at $t{=}0$, we make an explicit computation using the minimal principal series of $G$.  
We have seen in \eqref{eq-inf-ch-of-principal-series} that for the minimal parabolic subgroup $P= P_{\min}$ of $G$, 
\[
\InfCh (\pi_{\sigma, t^{-1}\nu}) = \InfCh(\sigma){\oplus} i t^{-1}\nu\in i \mathfrak{t}^*_{\min} {\oplus} i \mathfrak{a}^*_{\min} ,
\]
so that
\begin{equation} 
\label{eq-rescaled-in-ch} 
t \cdot \InfCh (\pi_{\sigma, t^{-1}\nu}) = t \cdot \InfCh(\sigma){\oplus} i  \nu\in i\mathfrak{t}^*_{\min} {\oplus} i \mathfrak{a}^*_{\min} .
\end{equation}
The representation $\pi_{\sigma,\nu}$ will  be irreducible if $\nu$ is  fixed by no nontrivial element of $W(\mathfrak{g},\mathfrak{a}_{\min})$, 
but in any case it always has an infinitesimal character. According to  Proposition~\ref{prop:continuous_induction}, the representations $\pi_{\sigma, t^{-1}\nu}$ of $G_t$ extend to a continuous family of representations with limit $\rho_{\sigma,\nu}$ on $G_0$. The continuous extension is therefore given on $\widehat G_0$ by
\begin{equation*}
 \rho_{\tau,\nu} \longmapsto 0{\oplus } i\nu \in i\mathfrak{t}^*_{\min} {\oplus} i \mathfrak{a}^*_{\min},
\end{equation*}
for every irreducible constituent $\rho_{\tau,\nu}$ of $\rho_{\sigma,\nu}$
Since every $\rho_{\tau,\nu}$ occurs in some $\rho_{\sigma, \nu}$ this completely determines the continuous extension to $t{=}0$ as claimed. 
\end{proof}

The imaginary part of the infinitesimal character always belongs to the image of the inclusion of $i\lie{a}^*_{\dom}$ into $\lie{h}^*_{\CC}/W(\lie{g}_{\C},\lie{h}_{\C})$; see Theorem \ref{thm-im-inf-ch-for-the-reduced-c-star-algebra}.  Therefore, composing the map of Theorem \ref{thm-continuous-inf-ch-on-family-dual} with the projection
\[
 \mathfrak{h}^*_{\C}/W(\mathfrak{g}_{\C},\mathfrak{h}_{\C})
 \longrightarrow 
 \Im (\mathfrak{h}^*_{\C})/W(\mathfrak{g}_{\C},\mathfrak{h}_{\C}),
\]
we obtain the the following.

\begin{theorem}
\label{thm-continuity-of-rescaled-im-inf-ch}
Let $G$ be a real reductive group.  There is a continuous map 
\[
\widehat{\GG}_{\mathrm{tempered}} \longrightarrow i   \mathfrak{a}_{\dom}^*
\]
such that 
\[
\widehat G_{t,\mathrm{tempered}} \ni \pi_t \longmapsto
t\cdot \ImInfCh({\pi}) \qquad (t{>}0) 
\]
and 
\[
\widehat G_{0,\mathrm{tempered}} \ni \pi_0 \longmapsto
 \ImInfCh({\pi}) \qquad (t{=}0) .
\]
\qed
\end{theorem}

\begin{definition}
    Given an irreducible representation $\pi$ of $C^*_r (\GG )$, we shall call the element of $i\lie{a}^*_{\dom}$ obtained by applying the map in Theorem \ref{thm-continuity-of-rescaled-im-inf-ch} to $\pi$ the \emph{scaled imaginary part of the infinitesimal character} of $\pi$.
\end{definition}
Thus for $t{>}0$ the scaled imaginary part of the infinitesimal character of the irreducible representation 
$\pi_t {=} \pi_{\sigma,t^{-1}\nu}$ of $G_t {=} G$ is $i \nu$,  while  the  scaled imaginary part of the infinitesimal character of the irreducible representation $\rho_{\tau,\nu}$ of $G_0$ is $i \nu$. 

\begin{remark}
    Similarly to $G$ and $G_0$, by the Dauns-Hofmann theorem there is a unique inclusion of $C_0(i \lie{a}^*_{\dom})$ into the center of the multiplier algebra of $C_r^*(\GG)$ such that 
    \[
    \pi_t \bigl ( (f\cdot a)(t)\bigr ) = f\bigl ( t \cdot \ImInfCh (\pi_t )\bigr ) \cdot \pi_t (a)
    \]
    for all $t{>}0$, and
     \[
    \pi_0 \bigl ( (f\cdot a)(0)\bigr ) = f( \ImInfCh (\pi_0 )) \cdot \pi_0 (a)
    \]
     for all irreducible  tempered unitary representations $\pi_t$ of $G_t$ and $\pi_0$ of $G_0$, all $f\in C_0(i \mathfrak{a}_{\dom}^*)$ and all $a\in C^* _r(\GG)$.
\end{remark}

\subsection{Subquotients of the C*-algebra of continuous sections}

\begin{definition}
\label{def-ideals-and-quotients-of-the-c-star-algebra-of-cts-sections}
We shall denote by 
    \[
    C^*_r(\GG ;U_I)\triangleleft C^*_r (\GG )
    \]
the $C^*$-algebra ideal   associated to the open subset of the tempered dual consisting of representations whose scaled infinitesimal character belongs to the open subset $U_I\subseteq \mathfrak{a}_{\dom}^*$.  We shall denote by 
    \[
    C^*_r(\GG ;U_I)\to C^*_r (\GG ; I)
    \]
the quotient of this ideal associated to the closed subset $\mathfrak{a}^*_{I,+}\subseteq U_I$.
\end{definition}

\begin{remark}
If $\{\, A_t \,\} _{t\in T}$ is any continuous field of $C^*$-algebras, and if, for every $t\in T$, $J_t$ is a $C^*$-algebra ideal in $A_t$, then $\{\, J_t\,\}_{t\in T}$ becomes a continuous field of $C^*$-algebras whose continuous sections are those of the original field that take values in the ideals $J_t$, provided that for every $t_0$ and for every $a\in J_{t_0}$ there is a continuous section of the original continuous field that takes values in the ideals $J_t$ and takes the value $a$ at $t_0$. This latter condition is not automatic, but it is  satisfied for the fields of ideals $C^*_r (G_t;U_I)$.
\end{remark}

\begin{remark}
We may also ask whether or not the field of quotients $\{ \,A_t/J_t\,\}_{t\in T}$ becomes a continuous field of $C^*$-algebras, if one deems the continuous sections to be those that lift to continuous sections of the original field.  This is not necessarily so, but the field is at least upper semicontinuous, in the sense that for every liftable section $s$, and every $t_0\in T$
\[
\limsup_{t\to t_0} \|s(t)\|_{A_t/J_t} \le \| s(t_0)\|_{A_{t_0}/J_{t_0}}.
\]
In fact the fields of quotients introduced in this section \emph{are} continuous fields, and not just upper semicontinuous fields.  We shall not need this fact, but it follows from a minor supplementary argument; see Remark~\ref{rem-injectivity-of-alpha-0} below.
\end{remark}

\subsection{Structure of the subquotients}

We shall now determine the structure of the subquotients $C^*_r (\GG, I)$. The formulation of the result and the method of proof are borrowed from \cite{HigsonRoman20}.

\begin{lemma}
\label{lem-rescaling-automorpshisms}
Let $G$ be a real reductive group and let $I$ be a subset of the set $S$ of simple roots in $\Delta^+(\mathfrak{g}, \mathfrak{a}_{\min})$. 
    There is a unique continuous action  of the multiplicative group $\R^{\times}_{+}$ on $C^*_r (G;I)$ by $C^*$-algebra automorphisms.
    \[
    \alpha_t\colon C_r^* (G,I) \longrightarrow C^*_r (G, I)\quad (t> 0)
    \]
   such that 
   \[
   \pi_{\sigma, \nu}(\alpha_t (f)) = \pi_{\sigma, t^{-1}\nu} (f)
   \]
   for all $t>0$, all $f\in C^*_r (G_t;I)$, all $\sigma\in \widehat M_{I,\mathrm{tempiric}}$, and all $\nu \in \mathfrak{a}^*_{I,+}$.
\end{lemma}

\begin{proof} 
The uniqueness assertion is clear, since the representations $\pi_{\sigma,\nu}$ separate the points of $C^*_r (G; I)$.
Existence follows from the isomorphism
\begin{equation}
    \label{eq-structure-of-face-algebra-recap}
    C^*_r (G;I)\stackrel \cong\longrightarrow \bigoplus _{\sigma\in \widehat M_{I,\mathrm{tempiric}}} C_0\bigl (\mathfrak{a}^*_{I}, \mathfrak{K} (\Ind_{P_I}^G H_\sigma)\bigr )
\end{equation}
in Theorem~\ref{thm-i-subquotient-of-g-isomorphism}, using which we may define $\alpha_t$ by requiring it to correspond with  the automorphisms
\[
\alpha_{\sigma,t} \colon C_0\bigl (\mathfrak{a}^*_{I}, \mathfrak{K} (\Ind_{P_I}^G H_\sigma)\bigr )
\longrightarrow C_0\bigl (\mathfrak{a}^*_{I}, \mathfrak{K} (\Ind_{P_I}^G H_\sigma)\bigr )
\]
of the summands on the right-hand side above defined by 
\[
\alpha_{\sigma,t} (h)(\nu) = h(t^{-1}\nu).
\qedhere\popQED
\]
\end{proof}

\begin{lemma}
\label{lem-mackey-embedding}
Let $G$ be a real reductive group and let $I$ be a subset of the set $S$ of simple roots in $\Delta^+(\mathfrak{g}, \mathfrak{a}_{\min})$.  There is a unique $C^*$-algebra homomorphism 
    \[
    \alpha_0\colon C^*_r (G_0;I) \longrightarrow C^*_r (G; I)
    \]
    such that 
    \[
    \pi_{\sigma, \nu} (\alpha_0(f_0)) = \rho_{\sigma,\nu}(f_0)
    \]
    for all $f_0\in C^*(G_0;I)$, all $\sigma\in \widehat M_{I,\mathrm{tempiric}}$,  and all $\nu\in \mathfrak{a}^*_{I,+}$.
\end{lemma}

\begin{proof} 
As with Lemma~\ref{lem-rescaling-automorpshisms}, the uniqueness assertion is clear, and to prove existence, we can again use the isomorphism \eqref{eq-structure-of-face-algebra-recap}, using which it suffices to construct a suitable $C^*$-homomorphism
\begin{equation}
    \label{eq-def-of-alpha-zero}
 \alpha_0\colon  C^*_r (G_0;I) \longrightarrow \bigoplus _{\sigma\in \widehat M_{I,\mathrm{tempiric}}} C_0\bigl (\mathfrak{a}^*_{I}, \mathfrak{K} (\Ind_{P_I}^G H_\sigma)\bigr ).
\end{equation}
We define 
\begin{equation}
    \label{eq-def-of-alpha-zero-fmla}
\alpha_0\colon 
f_0\longmapsto \oplus_\sigma \bigl [ \nu \mapsto \rho_{\sigma,\nu}(f_0)\bigr] 
\end{equation}
Considering an individual summand, if $f_0$ is a smooth and compactly supported function on $G_0$, and if $\varphi,\psi\in L^2(K, H_\sigma)^{K_I}$, then
\begin{equation}
\label{eq-matrix-coefficient-of-rep-of-motion-group}
\langle  \varphi , \rho_{\sigma,\nu} (f_0) \psi \rangle 
  = \int_K\int_K\int_{\lie{s}} 
 f_0\bigl (k'k^{-1}, \Ad_k X \bigr )  e^{-i \nu(X)}
  \bigl \langle \varphi(k'), \psi(k)\bigr \rangle \, dk'\, dk \, dX ,
\end{equation}
and the Riemann-Lebesgue lemma shows that the above is a $C_0$-function on $\nu$.
If $f_0$ is moreover $K$-bi-finite, then the operator $\rho_{\sigma,\nu} (f_0)$ is identically $0$ on a subspace of finite-codimension in  $L^2 (K,H_\sigma)^{K_I}$, and on the finite-dimensional orthogonal complement it is unitarily equivalent to a finite matrix with the above matrix coefficients \eqref{eq-matrix-coefficient-of-rep-of-motion-group} as entries. Therefore $\rho_{\sigma,\nu} (f_0)$ is a compact operator-valued $C_0$ function of $\nu$, for every $\sigma$. But in addition, $K$-bi-finiteness implies that only finitely many summands on the right of \eqref{eq-def-of-alpha-zero-fmla} are nonzero.   

It follows that the formula \eqref{eq-def-of-alpha-zero-fmla} does indeed define an element of the right-hand side in \eqref{eq-def-of-alpha-zero}, not just for a smooth, compactly supported and $K$-bi-finite function $f_0$  on $G_0$, but, by an approximation argument, for any $f_0\in C^*_r (G_0)$, since the functions we have considered up to now are dense in $C^*_r (G_0)$.
It is evident  the homomorphism $\alpha_0$ so-defined  has the required property.
\end{proof}

\begin{remark}
\label{rem-injectivity-of-alpha-0}
    The morphism $\alpha_0$ is injective.  This follows from the definition \eqref{eq-def-of-alpha-zero-fmla} above
    and the fact that the representations $\rho_{\sigma,\nu}$ include, as subrepresentations, all of the irreducible representations of $C^*(G_0;I)$ (in fact, by Frobenius reciprocity,  every irreducible representation of $K_I$ occurs as a $K_I$-type in some minimal principal series representation of $M_I$, and the base  principal series representations, with $\nu{=}0$, are tempiric).
\end{remark}

\begin{theorem}
\label{thm-limit-formula}
Let $G$ be a real reductive group and let $I$ be a subset of the set $S$ of simple roots in $\Delta^+(\mathfrak{g}, \mathfrak{a}_{\min})$. If 
 $\{ f_t\}$ is a continuous section of the continuous field of $C^*$-algebras $\{ C^*_r (G_t; U_I) \}$, then
\[
\lim_{t\searrow 0}  \alpha_t\bigl ( \pi_I(f_t) \bigr ) = \alpha_0 (\rho_I(f_0)) ,
\]
where the limit is taken  in  the norm topology of $C^*_r(G,I)$, and   where $\pi_I:C^*_r(G,U_I) \to C^*_r (G, I)$  and $\rho_I:C^*_r(G,U_I) \to C^*_r (G, I)$ are the quotient maps that were introduced in Definitions \textup{\ref{def-ideals-and-quotients-of-the-reduced-c-star-algebra}} and \textup{\ref{def-ideals-and-quotients-of-the-reduced-c-star-algebra-g-0}}.
\end{theorem}

\begin{remark}
If we apply the representation $\pi_{\sigma,\nu}$ to both sides of the limit formula in the theorem, and use the characteristic properties of $\alpha_t$ and $\alpha_0$, then we obtain the formula 
\[
\lim_{t\searrow 0}  \pi_{\sigma,t^{-1}\nu}(f_t)   = \rho_{\sigma,\nu}(f_0) .
\]
This is a consequence of Example~\ref{ex:dilation_family} and Proposition~\ref{prop:continuous_induction}.   In effect, Theorem~\ref{thm-limit-formula}   promotes this pointwise limit formula (valid for $\sigma$ and one value of $\nu$ at a time) to a  uniform limit formula  over all $\sigma$ and $\nu$.  This is a technical matter: the main conceptual content of the theorem is already present in Proposition~\ref{prop:continuous_induction}.
\end{remark}

\begin{proof}[Proof of Theorem~\ref{thm-limit-formula}]
Let $P_I = MAN$ be  the Langlands decomposition of   $P_I$.
Let $\{ f_t\}$ be a continuous section of the continuous field of $C^*$-algebras $\{ C^*_r (G_t; U_I) \}$. The same argument that we used in the proof of Lemma~\ref{lem-mackey-embedding} shows that it suffices to prove that for every tempiric representation $\sigma$, and for every pair $\varphi, \psi$ of smooth functions in  $ L^2 (K,H_\sigma)^{K\cap M}$, the limit formula 
\[
\lim_{t\searrow 0}  \langle \varphi , \pi_{\sigma,t^{-1}\nu}(f_t)\psi \rangle    = \langle \varphi , \rho_{\sigma,\nu}(f_0)\psi \rangle  
\]
holds uniformly in $\nu$.

Let $g\in G$.    Using a product decomposition 
\[
g = \kappa (g)\mu(g) \alpha(g) \upsilon (g) \in KMAN
\]
we may write the unitary operator 
\[
\pi_{\sigma,\nu}(g)\colon L^2 (K,H_\sigma)^{K\cap M} \to L^2 (K,H_\sigma)^{K\cap M} \]
in the form 
\begin{multline}
    \label{eq-compact-picture-of-parabolic-induction}
(\pi_{\sigma,\nu}(g) \psi)(\ell)
\\
= e^{-i \nu(\log(\alpha (g^{-1}\ell)))}e^{-\rho (\log (\alpha (g^{-1}\ell)))} \sigma(\mu(g^{-1}\ell ))^{-1}\psi(\kappa(g^{-1}\ell))
\end{multline}
for all $\ell \in K$. Compare \cite[Sec.~VII.1]{KnappRepTheorySemisimpleGroups}; note that the terms $\kappa(g^{-1}\ell)$ and $\mu(g^{-1}\ell)$ in the product decomposition of $g$ are not separately well-defined, but the   compound term $\sigma(\mu(g^{-1}\ell))^{-1}\psi(\kappa(g^{-1}\ell))$ is. It follows from \eqref{eq-compact-picture-of-parabolic-induction}, plus a change of variables, that if $t{>}0$, and if $f_t$ is a smooth and compactly supported function on $G_t{=}G$, then  
\begin{multline}
\label{eq-first-formula-for-pi-f-t}
(\pi_{\sigma,t^{-1}\nu}(f_t)\psi)(\ell) 
    \\
    = t^{-d} \int _G f_t(\ell g^{-1}) e^{-i t^{-1}\nu(\log(\alpha (g))}e^{-\rho (\log (\alpha (g))} 
    \\
   \times \sigma(\mu(g))^{-1}\psi(\kappa(g)) \, dg,
\end{multline}
where $d = \dim (G/K)$. Here we are using the smooth family of  Haar measures $t^{-d}dg$ on $G_t=G$ as in \eqref{eq-rescaled-Haar-measure}.  

Let us now fix an Iwasawa decomposition $M = K'A'N'$, with $K' = K\cap M$ and $A'\subseteq A_{\min}$.  The    integral in \eqref{eq-first-formula-for-pi-f-t} may be then written as a multiple integral
\begin{multline*}
 t^{-d}\int_K\int_{A'} \int_{N'}\int_A\int_N 
  f_t(\ell(ka'n'an)^{-1})  e^{-i t^{-1}\nu(\log(a))} e^{-\rho (\log (a))}
 \\
 \times \sigma(a'n')^{-1}\psi(k)  e^{+2(\rho' (\log (a'))+\rho(\log (a)))} \, dk \, da'\, dn'\, da\, dn ,
\end{multline*}
for suitable choices of Haar measure on the factor groups. See for instance  \cite[Sec.~V.6]{KnappRepTheorySemisimpleGroups} for this change of variables formula. It  follows that the inner product $ \langle  \varphi , \pi_{\sigma,t^{-1}\nu} (f_t) \psi \rangle $ is equal to 
\begin{multline}
    \label{eq-matrix-coefficient-for-g-t-0}
 t^{-d} \int _K \int_K\int_{A'}\int_{N'}\int_A\int_N 
  f_t (\ell(ka'n'an)^{-1})  e^{-it^{-1} \nu(\log(a))}
  \\
  \times \bigl \langle \varphi (\ell),\sigma(a'n')^{-1}\psi(k)\bigr \rangle  e^{2\rho' (\log (a'))+ \rho(\log (a))} \, d\ell \, dk \, da'\, dn'\, da\, dn .
\end{multline}

Let us make some  adjustments to \eqref{eq-matrix-coefficient-for-g-t-0}. The exponential maps for the groups $A'$, $N'$, $A$ and $N$ are measure-preserving diffeomorphisms (for suitable multiples of Lebesgue measure on the Lie algebras). So we may switch  to Lie algebra coordinates in the integrals over these groups.  Then, after scaling all four Lie algebra coordinates by a factor of $t$, we obtain the reformulation of the integral in \eqref{eq-matrix-coefficient-for-g-t-0}:
\begin{multline}
    \label{eq-matrix-coefficient-for-g-t-1}
  \int _K \int_K \int_{\mathfrak{a}'}\int _{\mathfrak{n}'} \int_{\mathfrak{a}}\int_{\mathfrak{n}} 
  \\
  f_t \bigl (\ell(k\exp(tX')\exp(tY')\exp(tX)\exp(tY))^{-1}\bigr  ) 
  \\
   \qquad \qquad \qquad \qquad \times   e^{-i  \nu(X)} \bigl \langle \varphi (\ell),\sigma (\exp(tX')\exp(tY'))^{-1}\psi(k)\bigr \rangle  
  \\
 \times  e^{+ t (2\rho' (X') + \rho (X))}d\ell \, dk\,   dX'\, dY'\, dX\, dY .
\end{multline}
The point about this reformulation is that the expression 
\[
k\exp(tX')\exp(tY')\exp(tX)\exp(tY) \in G_t
\]
that is included in the integrand above determines a diffeomorphism
\begin{equation}
    \label{eq-dnc-diffeomorphism}
K \times \mathfrak{a}'\times \mathfrak{n}'\times \mathfrak{a}\times \mathfrak{n} \times \R \stackrel \cong \longrightarrow 
\GG
\end{equation}
for which 
\[
(k,X',Y',X,Y,0) \longmapsto (k, X'{+}Y'{+}X{+}Y)\in G_0 = K \ltimes (\lie{g}/\lie{k}).
\]
So if the functions $f_t$ on the individual $G_t$ ($t{>}0$) assemble and extend  to a smooth, compactly supported function  on $\GG$, then the integrand in \eqref{eq-matrix-coefficient-for-g-t-1} is uniformly compactly supported on $K{\times} \mathfrak{a}'{\times}\mathfrak{n}'{\times}\mathfrak{a}{\times}\mathfrak{n} $ as $t{> }0$ varies, and moreover as $t$ tends to $0$ it converges uniformly to 
\[
 f_0 \bigl (\ell(k, X'{+}Y'{+}X{+}Y)^{-1}\bigr  )    e^{-i  \nu(X)} \bigl \langle \varphi (\ell), \psi(k)\bigr \rangle
\]
(note that extending the functional $\nu$, defined on $\mathfrak{a}$, to a linear functional on $\mathfrak{g}/\mathfrak{k}$ that vanishes on $\mathfrak{a'}$, $\mathfrak{n'}$ and $\mathfrak{n}$ is the same as extending it to a linear functional on $\mathfrak{s}$ that vanishes on the orthogonal complement of $\mathfrak{a}$, and then identifying $\mathfrak{s}$ with $\mathfrak{g}/\mathfrak{k}$). This gives  convergence of the integral in \eqref{eq-matrix-coefficient-for-g-t-1}  to 
\begin{equation}
\label{eq-matrix-coefficient-of-rep-of-g-t}
 \int_K\int_K\int_{\mathfrak{a}'+\mathfrak{n}'+\mathfrak{a}+\mathfrak{n}}
  f_0\bigl (\ell(k,Z)^{-1}\bigr )  e^{-i \nu(Z)}
  \bigl \langle \varphi(\ell), \psi(k)\bigr \rangle \, d\ell\, dk \, dZ ,
\end{equation}
which, as we have seen in \eqref{eq-matrix-coefficient-of-rep-of-motion-group}, is $\langle \varphi, \rho_{\sigma, \nu} (f_0)\psi\rangle$, and the convergence is uniform in $\nu$, as required.
\end{proof}

With this rather technical computation completed, we can readily determine the structure of  the $C^*$-algebra $C^*_r (\GG; I)$ in a way that will be useful for the $K$-theory computations of the next section.

\begin{definition}
 \label{def-mapping-cone-algebra}
 Let $\alpha:A_0 \rightarrow A$ be a  $C^*$-algebra homomorphism.   The \emph{mapping cone}  $C^*$-algebra is 
 \[
 \operatorname{Cone}(\alpha) =
 \bigl \{\,
(f_0,F)\in A_0\oplus C_0([0,\infty), A) : \alpha(f) = F(0)
 \,\bigr \} .
 \]
\end{definition}

\begin{theorem}[Compare  {\cite[Thm.~5.2.2]{HigsonRoman20}}]
\label{thm-cstar-iso-to-mapping-cone}
   Let $\alpha_0\colon C^*_r (G_0;I)\to C^*_r (G;I)$ be the $C^*$-algebra morphism from Lemma~\textup{\ref{lem-mackey-embedding}}. The formula 
    \[    \alpha_I(f) = \bigl ( f_0, t \mapsto \alpha_t(f_t) \bigr )\qquad (f\in C^*_r (\GG ; I) )
    \]
    defines a $C^*$-algebra isomorphism
    \[
    \alpha_I \colon C^*_r (\GG ;I) \stackrel \cong \longrightarrow \operatorname{Cone}(\alpha_0) .
    \]
\end{theorem}

\begin{proof}
Theorem~\ref{thm-limit-formula} tells us that the morphism $\alpha_I$ is well-defined, in that $t\mapsto \alpha_t (f_t)$ really is a $C_0$-function on $[0,\infty)$ with values in $C^*_r (G;I)$.  That $\alpha_I$ is an  isomorphism follows from the commuting diagram
\[
\xymatrix{
0 \ar[r] & C_0\bigl ( (0,\infty), C^*_r (G;I) \bigr ) \ar[d]^\cong  _{\{\alpha_t\}}\ar[r]& 
C^*_r (\GG; I) \ar[d]_{\alpha_I}\ar[r] & C^*_r (G_0;I) \ar[r] \ar@{=}[d]& 0
\\
0 \ar[r] & C_0\bigl ( (0,\infty), C^*_r (G;I) \bigr ) \ar[r]& 
\operatorname{Cone}(\alpha_0)  \ar[r] & C^*_r (G_0;I) \ar[r] & 0
}
\]
with exact rows, in which $\{ \alpha_t\}$ indicates the morphism that maps a function $t\mapsto f_t$ to the function $t\mapsto \alpha_t (f_t)$.
\end{proof}

\section{Vogan's theorem in K-theory}
\label{sec-vogans-thm}

In this section we shall formulate and prove our main theorem, which frames the Connes-Kasparov isomorphism in operator $K$-theory in purely representation-theoretic terms.

\subsection{Vogan's theorem on minimal K-types}

\begin{definition}[See {\cite[Def.\,5.4.18]{Voganbook}}] 
Let $G$ be a real reductive group with maximal compact subgroup $K$.  If $\tau$ is an irreducible unitary representation of $K$, then   its \emph{norm} is 
\[
\| \tau\| = \| \mu + 2 \rho_K\| ,
\]
where $\mu$ is a highest weight for some choice of maximal torus $T\subseteq K$ and some choice of a system of positive roots in $\Delta (\mathfrak{k}_{\C},\mathfrak{t}_{\C})$; the choices do not affect the norm. 
\end{definition}

\begin{definition}
\label{def-min-k-type}
Let $\pi$ be a irreducible unitary (or simply admissible) representation of $G$.  A \emph{minimal $K$-type} of $G$ is any $K$-type of $\pi$ (that is, any irreducible constituent of the restriction of $\pi$ to $K$) whose norm is minimal among the norms of all $K$-types of $\pi$.  
\end{definition}

\begin{remark}
Since every norm-bounded subset of $\widehat K$   is finite, every admissible representation  has at least one, and at most finitely many minimal $K$-types.  
\end{remark}

The simplicity of these definitions belies their importance. We shall be concerned in this paper with the following  fundamental result about minimal $K$-types, due to Vogan, which, as we shall see, completely determines the tempered dual.

\begin{theorem}[Vogan, {\cite[Thm.~1.2]{Vogan07}}]
\label{thm-vogan-minimal-k-type-theorem}
Let $G$ be a real reductive group with maximal compact subgroup $K$.  Every tempiric representation has a unique minimal $K$-type, which occurs in the representation with multiplicity one.  The map that associates to a tempiric representation its minimal $K$-type is a bijection 
\[
\widehat G _{\mathrm{tempiric}}\stackrel \cong \longrightarrow \widehat K .
\]
\end{theorem}

Among other things, this theorem immediately gives a parametrization of the tempered dual of $G$ in terms of irreducible representations of compact subgroups, for we have, from \eqref{eq-description-of-tempered-dual-via-im-inf-ch} and Vogan's bijection,
\[
 \widehat G_{\mathrm{tempered}} \cong  \bigsqcup _{I\subseteq S}
    \widehat M_{I,\mathrm{tempiric}} \times i \mathfrak{a}^*_{I,+} \cong \bigsqcup _{I\subseteq S}
    \widehat K_{I} \times i \mathfrak{a}^*_{I,+} \cong \widehat{G}_{0}. 
\]
This is the Mackey bijection of Afgoustidis \cite{AfgoustidisMackeyBijection}. 

In the next section we shall prove that the Connes-Kasparov isomorphism in operator $K$-theory is equivalent to     a $K$-theoretic version of Vogan's theorem that involves the following concepts.

\begin{definition}
Let  $G$ be a real reductive group. Denote by $R_{\mathrm{tempiric}}(G)$ the free abelian group on the set of equivalence classes of tempiric representations of $G$.
\end{definition}

\begin{definition}
\label{def-multiplicity-homomorphism}
Denote by 
 \[
    \mathrm{mult} \colon  R(K) \stackrel \cong  \longrightarrow R_{\mathrm{tempiric}} (G)
    \]
the homomorphism of abelian groups defined by 
    \[
    [\tau]\longmapsto \sum_{\sigma\,\,\mathrm{tempiric}} \mult (\tau : \sigma)\cdot  [\sigma]
    \qquad \forall \tau \in \widehat K ,
    \]
where $\mult (\tau : \sigma)$ is the multiplicity with which $\tau$ occurs in the restriction of the tempiric representation $\sigma$ to $K$.
\end{definition}

\begin{remark}
The multiplicity homomorphism is well-defined in view of the fact, noted in Theorem~\ref{thm-full-uniform-admissibility}, that each irreducible representation of $K$ occurs as a $K$-type in at most finitely many pairwise inequivalent tempiric representations.  
Formally, $\mult$ is the transpose of the map induced by restriction of tempiric $G$-representations to $K$, although that map is not well-defined to $R(K)$ since its image contains infinite sums.
\end{remark}

It is a simple matter of solving linear equations to prove, using Vogan's minimal $K$-type theorem (Theorem~\ref{thm-vogan-minimal-k-type-theorem}), that the multiplicity homomorphism in Definition~\ref{def-multiplicity-homomorphism} is an isomorphism.  The main goal of this article is to show that it is also possible to prove that $\mult$ is an isomorphism without using Vogan's minimal $K$-type theorem but instead using the Connes-Kasparov isomorphism.

This is very close to a new proof of Vogan's minimal $K$-type theorem. For instance, if we were to assume that no $\tau\in \widehat K$ may occur as a minimal $K$-type of two inequivalent  tempiric representations, then elementary linear algebra would allow us to deduce
Vogan's theorem from the isomorphism of the multiplicity map:

\begin{lemma}
    Suppose that the multiplicity map is an isomorphism. Suppose, in addition, that each $K$-type occurs in at most one tempiric representation of $G$ as a minimal $K$-type. Then every tempiric representation has a unique minimal $K$-type, which occurs in the representation with multiplicity one, and the map that sends a tempiric representation to its minimal $K$-type is a bijection.\qed 
\end{lemma}

\subsection{The Connes-Kasparov isomorphism}

We turn, finally to the Connes-Kasparov isomorphism. 
In view of the aims of this paper, it is certainly important to note that Lafforgue \cite{LafforgueInventiones} has given a proof of the Connes-Kasparov isomorphism using techniques from $K$-theory and index theory, rather than representation theory.  On the other hand, representation-theoretic treatments of the Connes-Kasparov isomorphism can be found in \cite{Wassermann87,Higson08,AfgoustidisConnesKasparov,ClareHigsonSong24,Bradd23}.  

We shall use the following formulation of the Connes-Kasparov isomorphism, which is due to Connes, see   \cite[Prop.~9,~p.141]{ConnesNCG}.

\begin{theorem}[Connes-Kasparov isomorphism]
Let $G$ be a real reductive group. The $C^*$-algebra $C^*_r (\GG )$ in Definition~\textup{\ref{def-family-group-c-star-algebra}} has vanishing $K$-theory.
\end{theorem}

\begin{remark}
The morphism $C^*_r (\GG)\to C^*_r (G_0)$ given by  evaluation at $t{=}0$ leads to  an extension of $C^*$-algebras
\[
\xymatrix{
0\ar[r] & C_0\bigl ((0,\infty), C^*_r (G)\bigr) \ar[r] & C^*_r (\GG ) \ar[r] &
C^*(G_0) \ar[r]& 0
}
\]
and hence (using Bott periodicity) a six-term exact sequence in $K$-theory
\[
\xymatrix{
  K_1\bigl ( C^*_r (G)\bigr) \ar[r] & K_0\bigl (C^*_r (\GG )\bigr )  \ar[r] &
K_0(C^*(G_0)) \ar[d]
\\
  K_1\bigl ( C^* (G_0)\bigr) \ar[u] & K_1\bigl (C^*_r (\GG )\bigr )  \ar[l] &
K_0(C^*_r(G)) . \ar[l]
}
\]
The assertion that the $K$-theory of $C^*_r (\GG )$ vanishes is therefore equivalent to the assertion that the vertical maps in the six-term sequence are isomorphisms, and it is these that Connes considers in \cite{ConnesNCG}. 
\end{remark}

\subsection{K-theory of the mapping cone}

The main theorem of this paper is the following consequence of the Connes-Kasparov isomorphism for real reductive groups.  In fact, this statement is equivalent to the Connes-Kasparov isomorphism for real reductive groups.

\begin{theorem}
\label{thm-vogan-thm-in-k-theory}
    Let $G$ be a real reductive group and let $K$ be a maximal compact subgroup of $G$.  The morphism of abelian groups
    \begin{align*}
    \mult: R(K) &\stackrel \cong  \longrightarrow R_{\mathrm{tempiric}} (G),\\
    [\tau] &\longmapsto \sum_{\sigma} \mult (\tau : \sigma)\cdot  [\sigma]
    \qquad \forall \tau \in \widehat K,
    \end{align*}
    is an isomorphism.
\end{theorem}

To prove the theorem we shall relate the morphism in the statement to the induced maps in $K$-theory associated to the $C^*$-algebra homomorphisms 
\[
\alpha_0 \colon C^*_r (G_0;I)\longrightarrow C^*_r (G;I)\qquad (I \subseteq S)
\]
from Lemma~\ref{lem-mackey-embedding}.   The point is that this is a morphism between $C^*$-algebras whose $K$-theory groups are easily determined in representation theoretic terms.  First, from the isomorphism 
\[
  C^*_r(G_0; I) \stackrel\cong \longrightarrow \bigoplus_{\tau \in \widehat{K}_I} C_0(\mathfrak a^*_{I,+}, \mathfrak{K}(\Ind_{K_I}^K H_\tau))
\]
in Theorem~\ref{thm-structure-of-g-0-component-algebra} we obtain isomorphisms
\[
\begin{aligned}
K_*\bigl (  C^*_r(G_0; I) \bigr ) 
& \stackrel\cong \longrightarrow \bigoplus_{\tau \in \widehat{K}_I} K_*\bigl (  C_0(\mathfrak a^*_{I,+}, \mathfrak{K}(\Ind_{K_I}^K H_\tau))\bigr ) \\
& \,\, \cong\,\,  \bigoplus_{\tau \in \widehat{K}_I} K_*\bigl (  C_0(\mathfrak a^*_{I,+})\bigr ) \\
& \,\, \cong\,\,  \bigoplus_{\tau \in \widehat{K}_I} \Z ,
\end{aligned}
\]
using first the stability property of $K$-theory and then the Bott periodicity theorem, which implies that the $K$-theory of $C_0(\mathfrak a^*_{I,+})$ is the  infinite cyclic group $\Z$.  We fix a choice of generator for each $I$.  Secondly, using Theorem \ref{thm-i-subquotient-of-g-isomorphism} in a similar fashion, we have isomorphisms
\[
K_*\bigl (  C^*_r(G; I) \bigr ) 
 \stackrel\cong \longrightarrow    \bigoplus_{\sigma \in \widehat{M}_{I,\mathrm{tempiric}}} \Z .
\]
\begin{theorem}
\label{thm-k-theory-of-cstar-calculation}
There is a commutative diagram 
\[
\xymatrix@C=40pt{
K_*(C^*_r(G_0;I))\ar[d]_{\cong} \ar[r]^{\alpha_{0*}} & K_*(C^*_r (G;I))\ar[d]^{\cong}
\\
R(K_I) \ar[r]^-{\mult_I} & 
R_{\mathrm{tempiric}} (M_I)
}
\]
in which the vertical isomorphisms are the ones described above, and the bottom morphism is the multiplicity map for the group $M_I$:
\[
\mult_I:[\tau]\mapsto \sum_{\sigma\in \widehat M_{I,\mathrm{tempiric}}} \operatorname{mult}(\tau:\sigma)[\sigma] .
\]
\end{theorem}
\begin{proof}
Let $\tau \in \widehat{K}_I$ and let  $\sigma\in \widehat M_{I,\mathrm{tempiric}}$. We need to compute that the bottom  arrow in the diagram
\[
\xymatrix{
K_*(C_r^*(G_0; I)) \ar[d]_{\rho_{\tau*}} \ar[r]^{\alpha_{0,*}} & K_*(C_r^*(G ; I)) \ar[d]^{\pi_{\sigma*}}
\\
K_*\bigl(C_0(\mathfrak a^*_{I,+}, \mathfrak{K}(\Ind_{K_I}^K H_\tau))\bigr) \ar[d]_\cong \ar[r] & K_*\bigl(C_0(\mathfrak a^*_{I,+}, \mathfrak{K}(\Ind_{P_I}^G H_\sigma))\bigr) \ar[d]^\cong
\\
\mathbb{Z} \cdot [\tau] \ar[r] & \mathbb{Z} \cdot [\sigma].
}
\]
is multiplication by $\mult(\tau: \sigma)$.
For this, we shall calculate the map
\[C_0(\mathfrak a^*_{I,+}, \mathfrak{K}(\Ind_{K_I}^K H_\tau)) \longrightarrow C_0(\mathfrak a^*_{I,+}, \mathfrak{K}(\Ind_{P_I}^G H_\sigma))\]
which induces the middle arrow in $K$-theory. Let $F \in C_0(\mathfrak{a}^*_{I, +}, \mathfrak{K}(\Ind_{K_I}^K H_\tau))$, and let $f_0$ be the element of $C_r^*(G_0; I)$ such that $\rho_{\tau, \nu}(f) = F(\nu)$ and $\rho_{\tau', \nu}(f) = 0$ when $[\tau'] \neq [\tau]$. By definition, the composition
\[C_r^*(G_0; I) \to C_r^*(G;I) \to C_0(\mathfrak{a}^*_{I,+}, \mathfrak{K}(\Ind_{P_I}^G H_\sigma)) \]
sends $f_0$ to the function $\rho_{\sigma}(f_0)$ defined by $\rho_\sigma(f_0)(\nu) = \rho_{\sigma, \nu}(f_0)$. 

Now, we have the decomposition
\begin{align*}
L^2 (K,H_\sigma) &= \bigoplus_{\tau'\in\widehat{K}_I }\operatorname{mult}(\tau': \sigma)\cdot  L^2 (K,H_{\tau'}),\\
\rho_{\sigma,\nu}(f_0) &= \bigoplus_{\tau'\in\widehat{K}_I }\operatorname{mult}(\tau': \sigma)\cdot \rho_{\tau',\nu}(f_0),
\end{align*}
where $n \cdot T$ denotes the direct sum of $T$ with itself $n$ times. It follows that the middle arrow in the above diagram is induced by the map
\[F \mapsto \mult(\tau: \sigma) \cdot F,\]
sending a matrix to the direct sum of that matrix with itself $\mult(\tau: \sigma)$ times. In terms of $K$-theory, this corresponds to multiplication by $\mult(\tau: \sigma)$, once appropriate generators corresponding to $[\tau]$ and $[\sigma]$ are fixed.
\end{proof}

We now turn to the proof of Theorem \ref{thm-vogan-thm-in-k-theory}.  We will prove it simultaneously with the following theorem, by using the Connes-Kasparov isomorphism and induction  on the real rank of the group $G$ (note that $G {=} K$ when the real rank is $0$). 

\begin{theorem}
\label{thm-i-subquotient-of-fat-g-vanishes-in-k-theory}
Let $G$ be a real reductive group, and let $S$ be a set of simple roots for the restricted root system $\Delta(\mathfrak{g},\mathfrak{a}_{\min})$. 
For every $I\subseteq S$,    $K_*(C^*_r (\GG ;I)) =0$.
\end{theorem}

\begin{proof}[Proof of Theorems \ref{thm-vogan-thm-in-k-theory} and \ref{thm-i-subquotient-of-fat-g-vanishes-in-k-theory}]

Suppose that Theorem \ref{thm-vogan-thm-in-k-theory} holds for groups with real rank strictly less than $r$, and suppose that $G$ has real rank $r$.  
Firstly, suppose that $I \subset S$ is a proper subset of $S$. As the group $M_I$ has lower real rank, the morphism
\[R(K_I) \to R_{\mathrm{tempiric}}(M_I)\]
is an isomorphism by the inductive hypothesis. By Theorem \ref{thm-k-theory-of-cstar-calculation}, this means that $\alpha_{0}: C^*(G_0;I) \to C^*_r(G;I)$ induces an isomorphism on $K$-theory, so by Theorem \ref{thm-cstar-iso-to-mapping-cone}, we obtain that $K_*(C^*_r(\GG;I)) = 0$ for $I\subset S$ proper.

Now, for each $0 \leq k \leq r$, define 
\[U_k = \bigcup_{|I| = k} U_I,\]
where $|I|$ denotes the cardinality of $I$. Define $C_r^*(\GG; U_k)$ as the $C^*$-algebra corresponding to the open set $U_k$. Then $U_k \setminus U_{k-1}$ is the disjoint union of closed subsets
\[U_k \setminus U_{k-1} = \bigsqcup_{|I| = k} \mathfrak{a}^*_{I,+},\]
and so we have the exact sequence 
\[ 0 \to C_r^*(\GG;U_{k-1}) \to C_r^*(\GG ; U_k) \to \bigoplus_{|I| = k}  C_r^*(\GG; I) \to 0.\]
It therefore follows from the $6$-term exact sequence in $K$-theory and induction on $k$ that $K_*(C_r^*(\GG; U_k)) = 0$ for $k < r$.

Now consider $I = S$. In this case, we note that
\[U_r \setminus U_{r-1} = \mathfrak{a}^*_{S,+},\] 
and $C_r^*(\GG; U_r) = C_r^*(\GG)$, so we have an exact sequence
\[0 \to C_r^*(\GG; U_{r-1}) \to C_r^*(\GG) \to  C_r^*(\GG; S) \to 0.\]
As both $C_r^*(\GG; U_{r-1})$ and $C_r^*(\GG)$ vanish in $K$-theory (this is where we use Connes-Kasparov isomorphism), so must $C_r^*(\GG; S)$, which proves Theorem \ref{thm-i-subquotient-of-fat-g-vanishes-in-k-theory} for $G$. By Theorem \ref{thm-k-theory-of-cstar-calculation}, this also proves Theorem \ref{thm-vogan-thm-in-k-theory} for $G$.
\end{proof}

We have used the Connes-Kasparov isomorphism to prove Theorem \ref{thm-vogan-thm-in-k-theory}. Conversely, by arguing exactly as in the above proof, we can prove that $C_r^*(\GG)$ has vanishing $K$-theory if we are given Theorem \ref{thm-vogan-thm-in-k-theory}. That is, we have the following.

\begin{theorem}
The assertion that the multiplicity map in Definition~\textup{\ref{def-multiplicity-homomorphism}} is an isomorphism, for every real reductive group, is equivalent to the Connes-Kasparov isomorphism for every real reductive group. \qed
\end{theorem}

\appendix
\section{Appendices} 
\subsection{Consequences of the Connes-Kas\-parov iso\-morphism}
\label{appendix-connes-kasparov}

Lafforgue used the Connes-Kasparov isomorphism to obtain a new proof in \cite{LafforgueICM} of Harish-Chandra's equal rank criterion for the existence of discrete series representations of a   connected real reductive group with compact center \cite[Thm.\,13]{HarishChandra66}, along with a new proof of Harish-Chandra's parametrization of the discrete series representations \cite[Thm.\,16]{HarishChandra66}, when they exist.  The purpose of this section is simply to recall that following result also follows from Harish-Chandra's classificiation of the discrete series, and therefore  may also be viewed as consequences of the Connes-Kasparov isomorphism:

\begin{theorem*}[Theorem~\ref{thm-uniform-admissibility-for-discrete-series} in this text]
Let $G$ be a real reductive group with compact center that admits discrete series representations and let $K$ be a maximal compact subgroup of $G$. Each irreducible representation of $K$ occurs as a $K$-type in at most finitely many discrete series representations of $G$, up to unitary equivalence.
\end{theorem*}

\begin{proof} Assume first that $G$ is a connected group that admits discrete series.  Let $T$ be a maximal torus in $K$.  According to Harish-Chandra, the discrete series are parametrized by their \emph{Harish-Chandra parameters}, which are certain classes   
\[ 
[\lambda] \in i\mathfrak{t}^* / W(\mathfrak{k}_{\C},\mathfrak{t}_{\C})
\]
such that $\lambda + \rho_G$ is analytically integral.   The  character of the discrete series representation $\pi_\lambda$ associated to $\lambda$ is the  class 
\begin{equation}
\label{eq-inf-ch-of-a-discrete-series}
[\lambda] \in  \mathfrak{t}^*_{\C} / W(\mathfrak{g}_{\C},\mathfrak{t}_{\C})
\end{equation}
(according to Harish-Chandra's criterion, discrete series representations exist for $G$ if and only if $\mathfrak{t}$ by itself is a Cartan subalgebra of $\mathfrak{g}$). It therefore follows from a well-known calculation (\emph{c.f.}~\cite[Prop.\,5.28]{KnappBeyond}) that  if 
\[
\Omega_{\mathfrak{g}} = \sum X^aX_a
\]
is the Casimir element in the center of the enveloping algebra for $\mathfrak{g}$ (where $\{ X_a\}$ is basis for $\mathfrak{g}$, and $\{ X^a\}$ is a dual basis for the invariant bilinear form $B$ in \eqref{eq-def-of-b}), and if $\InfCh_{\Omega_{\mathfrak{g}}}$ is the function on $\widehat G$ from Definition~\ref{def-inf-ch}, then 
\[
\InfCh_{\Omega_{\mathfrak{g}}} (\pi_\lambda) = \| \lambda\|^2 - \|\rho_G\|^2 .
\]
 Now write the Casimir as 
\[
\Omega_{\mathfrak{g}} = \Delta_{\mathfrak{s}} + \Omega_{\mathfrak{k}} =  \sum_b Y_b^2   - \sum_c Z_c^2 ,
\]
where $\{ Y_b\}$ is an orthonormal basis for $\mathfrak{s}$ and $\{ Z_c\}$ is an orthonormal basis for $\mathfrak{k}$ in the Cartan decomposition $\mathfrak{g} =  \mathfrak{k}\oplus \mathfrak{s}$.  It follows from the unitarity of $\pi_\lambda$  that if     $\tau$ is a irreducible unitary representation of $K$ with infinitesimal character $\mu\in i\mathfrak{t}^* / W(\mathfrak{k}_{\C},\mathfrak{t}_{\C})$, and if $v$ is a unit vector in the $\tau$-isotypical part of $H_\pi$, then 
\[
 \|\lambda\|^2 = \langle v, \Omega_{\mathfrak{g}} v\rangle = \langle v, \Delta_{\mathfrak{s}} v\rangle +  \|\mu\|^2 -\|\rho_K\|^2 \le   \|\mu\|^2 - \|\rho_K\|^2.
\]
It follows that only finitely many $\pi_\lambda$ can include $\tau$ as a $K$-type, which proves Theorem~\ref{thm-uniform-admissibility-for-discrete-series} in the connected case.  

The general case follows from the connected case by Frobenius reci\-procity, since the restriction of a discrete series representation of $G$ to the connected component of $G$ is a direct sum of discrete series representations, all of them sharing the same infinitesimal character.
\end{proof}

 \begin{remark}
     \label{remark-appendix-harish-chandra-parameters}
     The formula \eqref{eq-inf-ch-of-a-discrete-series} shows immediately that every discrete series representation has real infinitesimal character: recall from Definition~\ref{def-real-and-imaginary-weights} that if $\mathfrak{h}= \mathfrak{t} {\oplus} \mathfrak{a}$ is any $\theta$-stable Cartan subalgebra of $\mathfrak{g}$, then   a weight $\lambda\in \mathfrak{h}_{\C}^*$ is   real if it is imaginary-valued on $\mathfrak{t}$ and real-valued on $\mathfrak{a}$, and this is the case here, where in fact $\mathfrak{a} = 0$.
 \end{remark}

\subsection{Classification of tempered irreducible unitary  representations by tempiric representations} 
\label{sec-proof-of-classification}

We shall sketch a proof  of Vogan's Theorem~\ref{thm-vogan-characterization-of-unitary-reps-with-fixed-im-inf-ch}, which classifies tempered irreducible unitary representations in terms of the imaginary part of the infinitesimal character and tempiric representations Levi subgroups.  For convenience, we restate the result here.  Recall that we had fixed an Iwasawa decomposition $G=KA_{\min}N$ of a real reductive group and that we are denoting by $S\subseteq \Delta(\mathfrak{g},\mathfrak{a}_{\min})$ the associated set of simple restricted roots.  Here, again, is the classification theorem:

\begin{theorem*}[Theorem~\ref{thm-vogan-characterization-of-unitary-reps-with-fixed-im-inf-ch} in this 
text]
Let $I\subseteq S$   and let  $\nu\in  \mathfrak{a}_{I,+}^*$. The correspondence 
\[
\sigma \longmapsto \Ind_{P_I}^G \sigma {\otimes} \exp(i \nu) 
\]
determines a bijection from the set of unitary  equivalence classes of tempiric representations $\sigma$ of $M_I$  to the set of unitary  equivalence classes of irreducible tempered unitary representations $\pi$ of $G$ for which 
$\ImInfCh(\pi) = \nu$.
\end{theorem*}

We shall use Harish-Chandra's Theorem~\ref{thm-harish-chandra-principle-on-cuspidal-principal-series} in the proof, along with the following result, which is also due to Harish-Chandra.

\begin{theorem}
\label{thm-harish-chandra-irreducibility}
    Let $P=MAN$ be a parabolic subgroup of $G$.  Let $\nu\in \mathfrak{a}^*$, and assume that the Levi subgroup $L=MA$ is the full centralizer of $\nu$ in $G$. If  $\sigma$ is an irreducible unitary representation of $M$ with real infinitesimal character,  then the parabolically induced representation $\Ind_P^G \sigma\otimes \exp (i \nu)$ is irreducible.  Moreover, if $\sigma_1$ and $\sigma_2$ are irreducible unitary representations of $M$ with real infinitesimal character, and if the representations  $\Ind_P^G \sigma_1\otimes \exp (i \nu)$ and  $\Ind_P^G \sigma_2\otimes \exp (i \nu)$ are equivalent, then $\sigma_1$ and $\sigma_2$ are equivalent.
\end{theorem}

\begin{remark}
In the case of the minimal parabolic subgroup Theorem~\ref{thm-harish-chandra-irreducibility} was proved by Bruhat \cite{Bruhat}. Harish-Chandra developed Bruhat's method to yield the general result above, but did not publish his result.  See \cite{KolkVaradarajan96} for further historical remarks, and for an exposition of most of the proof. See \cite[Thm.\,10.7]{VanDenBanSchlichtkrull05} for the remaining details.
\end{remark}

\begin{proof}[Proof of Theorem~\ref{thm-vogan-characterization-of-unitary-reps-with-fixed-im-inf-ch}]
First, it follows from Theorem~\ref{thm-harish-chandra-irreducibility} that all the representations $\Ind_P^G \sigma_1{\otimes} \exp (i \nu)$ listed in the statement of the theorem are irreducible, and it follows from \eqref{eq-inf-ch-of-principal-series} that they all have imaginary part of the infinitesimal character equal to $\nu$. 

Next let $\pi$ be a tempered, irreducible unitary representation of $G$.  According to Theorem~\ref{thm-harish-chandra-principle-on-cuspidal-principal-series}, $\pi$  may be embedded as a subrepresentation of some parabolically induced representation $\Ind_P^G \theta{\otimes}\! \exp(i\nu)$, where $P=MAN$ is a standard parabolic subgroup of $G$, $\theta$ is a discrete series representation of $M$, and $\nu\in \mathfrak{a}^*$.  

If we conjugate by a representative in $K$ of an appropriate element of $W(\mathfrak{g},\mathfrak{a}_{\min})$ then we may arrange that $\nu \in \mathfrak{a}_{\dom}^*$ (recall our convention that we extend $\nu$ to $\mathfrak{a}_{\min}$ by requiring it to be zero on the orthogonal complement of $\mathfrak{a}\subseteq \mathfrak{a}_{\min}$).  The conjugation may take $P$ to a nonstandard parabolic subgroup, but we may replace the conjugated parabolic subgroup by a standard parabolic subgroup in the same associate class without changing  the equivalence class of the parabolically induced representation.

In summary, we may in embed 
$\pi$ as a subrepresentation of some parabolically induced representation $\Ind_P^G \theta\otimes \exp(i\nu)$, where $P=MAN$ is a standard parabolic subgroup of $G$, $\theta$ is a discrete series representation of $M$, and $\nu\in \mathfrak{a}^*\cap \mathfrak{a}_{\dom}^*$.  Of course, $\nu$ is the imaginary part of the infinitesimal character of $\pi$.

Now  let $I\subseteq S$ be such that $\nu\in \mathfrak{a}_{I,+}^*$. The group  $L{=}MA$ is a Levi subgroup of $L_I {=} M_IA_I$, and if we parabolically induce $\theta \otimes\exp (i \nu)$ from $L$ to $L_I$ we obtain a representation of the form 
$\sigma \otimes \exp(i \nu)$.  The representation $\sigma$ is not necessarily irreducible, but it is a direct sum of tempiric representations of $M_I$,
\[
\sigma = \sigma_1\oplus\cdots \oplus \sigma_d.
\]
By induction in stages, $\pi$ embeds in some parabolically induced representation $\Ind_{P_I}^G \sigma_k\otimes \exp(i \nu)$.  It follows from Theorem~\ref{thm-harish-chandra-irreducibility} that the parabolically induced representation is irreducible, and therefore the embedding of $\pi$ into it must be a unitary  equivalence.  It also follows from Theorem~\ref{thm-harish-chandra-irreducibility} that the unitary equivalence class of $\pi$ determines the unitary equivalence class of $\sigma_k$. This completes the proof of the theorem.
\end{proof}

\bibliographystyle{alpha}
\bibliography{refs}

\def\cprime{$'$} \def\cprime{$'$}
\begin{thebibliography}{CHST24}

\bibitem[Afg19]{AfgoustidisConnesKasparov}
A.~Afgoustidis.
\newblock On the analogy between real reductive groups and {C}artan motion
  groups: a proof of the {C}onnes-{K}asparov isomorphism.
\newblock {\em J. Funct. Anal.}, 277(7):2237--2258, 2019.

\bibitem[Afg21]{AfgoustidisMackeyBijection}
A.~Afgoustidis.
\newblock On the analogy between real reductive groups and {C}artan motion
  groups: the {M}ackey-{H}igson bijection.
\newblock {\em Camb. J. Math.}, 9(3):551--575, 2021.

\bibitem[Bra23]{Bradd23}
J.~Bradd.
\newblock Compatible decomposition of the {C}asselman algebra and the reduced
  group {C*}-algebra of a real reductive group.
\newblock arXiv:2312.11773, 2023.

\bibitem[Bru56]{Bruhat}
F.~Bruhat.
\newblock Sur les repr\'esentations induites des groupes de {L}ie.
\newblock {\em Bull. Soc. Math. France}, 84:97--205, 1956.

\bibitem[CCH16]{CCH16}
P.~Clare, T.~Crisp, and N.~Higson.
\newblock Parabolic induction and restriction via {$C^*$}-algebras and
  {H}ilbert {$C^*$}-modules.
\newblock {\em Compos. Math.}, 152(6):1286--1318, 2016.

\bibitem[Che55]{Chevalley55}
C.~Chevalley.
\newblock Invariants of finite groups generated by reflections.
\newblock {\em Amer. J. Math.}, 77:778--782, 1955.

\bibitem[CHR24]{ClareHigsonRoman24}
P.~Clare, N.~Higson, and A.~Rom\'an.
\newblock A {M}ackey embedding for reduced {C*}-algebras of real reductive
  groups.
\newblock In preparation, 2024.

\bibitem[CHS24]{ClareHigsonSong24}
P.~Clare, N.~Higson, and Y.~Song.
\newblock On the {C}onnes-{K}asparov isomorphism, {II}.
\newblock {\em Jpn. J. Math.}, 19(1):111--141, 2024.

\bibitem[CHST24]{ClareHigsonSongTang24}
P.~Clare, N.~Higson, Y.~Song, and X.~Tang.
\newblock On the {C}onnes-{K}asparov isomorphism, {I}.
\newblock {\em Jpn. J. Math.}, 19(1):67--109, 2024.

\bibitem[Con94]{ConnesNCG}
A.~Connes.
\newblock {\em Noncommutative {G}eometry}.
\newblock Academic Press, San Diego, 1994.

\bibitem[DH68]{DaunsHoffman68}
J.~Dauns and K.~H. Hofmann.
\newblock {\em Representation of rings by sections}.
\newblock Memoirs of the American Mathematical Society, No. 83. American
  Mathematical Society, Providence, RI, 1968.

\bibitem[DH24]{DeBelloHigson24}
P.~DeBello and N.~Higson.
\newblock Pseudodifferential operators and the {C}onnes-{K}asparov isomorphism.
\newblock In preparation, 2024.

\bibitem[Dix77]{DixmierEnglish}
J.~Dixmier.
\newblock {\em {$C\sp*$}-algebras}.
\newblock North-Holland Publishing Co., Amsterdam-New York-Oxford, 1977.
\newblock Translated from the French by Francis Jellett, North-Holland
  Mathematical Library, Vol. 15.

\bibitem[HC66]{HarishChandra66}
Harish-Chandra.
\newblock Discrete series for semisimple {L}ie groups. {II}. {E}xplicit
  determination of the characters.
\newblock {\em Acta Math.}, 116:1--111, 1966.

\bibitem[Hel78]{HelgasonDS}
S.~Helgason.
\newblock {\em Differential geometry, {L}ie groups, and symmetric spaces},
  volume~80 of {\em Pure and Applied Mathematics}.
\newblock Academic Press, Inc. [Harcourt Brace Jovanovich, Publishers], New
  York-London, 1978.

\bibitem[Hig08]{Higson08}
N.~Higson.
\newblock The {M}ackey analogy and {$K$}-theory.
\newblock In {\em Group Representations, Ergodic Theory, and Mathematical
  Physics: A Tribute to {G}eorge {W}. {M}ackey}, volume 449 of {\em Contemp.
  Math.}, pages 149--172. Amer. Math. Soc., Providence, RI, 2008.

\bibitem[HR20]{HigsonRoman20}
N.~Higson and A.~Rom\'{a}n.
\newblock The {M}ackey bijection for complex reductive groups and continuous
  fields of reduced group {C}*-algebras.
\newblock {\em Represent. Theory}, 24:580--602, 2020.

\bibitem[Kna86]{KnappRepTheorySemisimpleGroups}
A.~W. Knapp.
\newblock {\em Representation theory of semisimple groups}, volume~36 of {\em
  Princeton Mathematical Series}.
\newblock Princeton University Press, Princeton, NJ, 1986.
\newblock An overview based on examples.

\bibitem[Kna02]{KnappBeyond}
A.~W. Knapp.
\newblock {\em Lie groups beyond an introduction}, volume 140 of {\em Progress
  in Mathematics}.
\newblock Birkh\"{a}user Boston, Inc., Boston, MA, second edition, 2002.

\bibitem[KV96]{KolkVaradarajan96}
J.~A.~C. Kolk and V.~S. Varadarajan.
\newblock On the transverse symbol of vectorial distributions and some
  applications to harmonic analysis.
\newblock {\em Indag. Math. (N.S.)}, 7(1):67--96, 1996.

\bibitem[Laf02a]{LafforgueICM}
V.~Lafforgue.
\newblock Banach {KK}-theory and the {B}aum-{C}onnes conjecture.
\newblock In {\em Proceedings of the ICM 2002 (Beijing)}, volume~II, pages
  795--812, 2002.

\bibitem[Laf02b]{LafforgueInventiones}
V.~Lafforgue.
\newblock K-th\'eorie bivariante pour les alg\`ebres de {B}anach et conjecture
  de {B}aum-{C}onnes.
\newblock {\em Invent. Math.}, 149(1):1--95, 2002.

\bibitem[Lan89]{Langlands89}
R.~P. Langlands.
\newblock On the classification of irreducible representations of real
  algebraic groups.
\newblock In {\em Representation theory and harmonic analysis on semisimple
  {L}ie groups}, volume~31 of {\em Math. Surveys Monogr.}, pages 101--170.
  Amer. Math. Soc., Providence, RI, 1989.

\bibitem[Mac49]{Mackey49}
G.~W. Mackey.
\newblock Imprimitivity for representations of locally compact groups. {I}.
\newblock {\em Proc. Nat. Acad. Sci. U. S. A.}, 35:537--545, 1949.

\bibitem[Ped79]{Pedersen79}
G.~K. Pedersen.
\newblock {\em {$C^{\ast} $}-algebras and their automorphism groups}, volume~14
  of {\em London Mathematical Society Monographs}.
\newblock Academic Press, Inc. [Harcourt Brace Jovanovich, Publishers],
  London-New York, 1979.

\bibitem[vdBS05]{VanDenBanSchlichtkrull05}
E.~P. van~den Ban and H.~Schlichtkrull.
\newblock The {P}lancherel decomposition for a reductive symmetric space. {II}.
  {R}epresentation theory.
\newblock {\em Invent. Math.}, 161(3):567--628, 2005.

\bibitem[Vog81]{Voganbook}
D.~A. Vogan, Jr.
\newblock {\em Representations of real reductive {L}ie groups}, volume~15 of
  {\em Progress in Mathematics}.
\newblock Birkh\"{a}user, Boston, Mass., 1981.

\bibitem[Vog00]{Vogan00}
D.~A. Vogan, Jr.
\newblock A {L}anglands classification for unitary representations.
\newblock In {\em Analysis on homogeneous spaces and representation theory of
  {L}ie groups, {O}kayama--{K}yoto (1997)}, volume~26 of {\em Adv. Stud. Pure
  Math.}, pages 299--324. Math. Soc. Japan, Tokyo, 2000.

\bibitem[Vog07]{Vogan07}
D.~A. Vogan, Jr.
\newblock Branching to a maximal compact subgroup.
\newblock In {\em Harmonic analysis, group representations, automorphic forms
  and invariant theory}, volume~12 of {\em Lect. Notes Ser. Inst. Math. Sci.
  Natl. Univ. Singap.}, pages 321--401. World Sci. Publ., Hackensack, NJ, 2007.

\bibitem[Was87]{Wassermann87}
A.~Wassermann.
\newblock Une d\'emonstration de la conjecture de {C}onnes-{K}asparov pour les
  groupes de {L}ie lin\'eaires connexes r\'eductifs.
\newblock {\em C. R. Acad. Sci. Paris S\'er. I Math.}, 304(18):559--562, 1987.

\end{thebibliography}

\end{document}